\documentclass[twoside, 10pt]{amsart}
\usepackage[letterpaper,hmargin=1.00in,vmargin=1.25in]{geometry}
\usepackage{amscd}
\usepackage{verbatim}
\usepackage{color}
\usepackage{comment}
\input xy			
\xyoption{all}
\usepackage{amssymb}
\usepackage{hyperref}

\DeclareMathAlphabet{\cat}{OT1}{cmss}{m}{sl}
\newtheorem*{theorem*}{Theorem}
\newtheorem{theorem}{Theorem}[section]

\newtheorem{proposition}[theorem]{Proposition}
\newtheorem{lemma}[theorem]{Lemma}
\newtheorem{corollary}[theorem]{Corollary}

\theoremstyle{definition}
\newtheorem{remark}[theorem]{Remark}

\newcommand{\diag}{\operatorname{diag}}
\newcommand{\diam}{\operatorname{Diam}}
\newcommand{\mat}{\operatorname{M}}
\newcommand{\Ch}{\operatorname{char}}
\newcommand{\SL}{\operatorname{SL}}
\newcommand{\GL}{\operatorname{GL}}
\newcommand{\PSL}{\operatorname{PSL}}
\newcommand{\rank}{\operatorname{rank}}

\title[Commuting involution graphs of linear groups] 
{Commuting involution graphs of linear groups}

\author
[S. Baek] {Sanghoon Baek}

\author
[C. Han]{Changhyouk Han}

\address{Department of Mathematical Sciences, KAIST, 291 Daehak-ro, Yuseong-gu, Daejeon, 305-701, Republic of Korea}

\email {sanghoonbaek@kaist.ac.kr}
\email {vicent@kaist.ac.kr}
\begin{document}

\begin{abstract}
In this paper, we determine the diameter of the commuting involution graphs of special and general linear groups over an arbitrary field. It turns out that our results also determine the diameter for certain projective special linear groups over finite fields. Moreover, we find the diameter of the commuting graphs of general linear groups on the set of all involutions over a field of characteristic $2$, which completes the diameter of general linear groups on the set of all involutions. As an application, we classify the structure of the four-dimensional linear groups over finite fields according to the distance from a fixed involution. 
\end{abstract}
\maketitle

\tableofcontents

\section{Introduction}\label{intro}
Let $X$ be a subset of a group $G$. The \emph{commuting graph} of $G$ on $X$, denoted by $\Gamma(G, X)$, is the graph with the vertex set $X$ in which distinct elements $x, y\in X$ are connected by an edge if $xy=yx$. In particular, if $X$ is a conjugacy class of involutions, the corresponding graph is called the \emph{commuting involution graph}. The \emph{distance} $d(x, y)$ between $x$ and $y$ is the number of edges in a shortest path connecting them. The \emph{diameter} of a connected graph $\Gamma(G, X)$, denoted by $\diam(\Gamma(G, X)$), is the maximum distance over all pairs of vertices.

The commuting graph on the set of involutions was first studied by Brauer and Fowler in \cite{BF}, where they show that the commuting graph of a finite group is connected with diameter at most $3$ if it has more than one conjugacy class.  This result is one of the key ingredients in the recent work of Morgan and Parker \cite{MP} on an upper bound of $\diam(\Gamma(G, X))$ of a finite centerless group $G$ on $X=G\backslash \{1\}$. Indeed, this result can be viewed as a generalization of Segev and Seitz's result \cite{SS} on finite simple classical groups, which in turn have played a major role in their proof of the Margulis-Platanov conjecture for
inner forms of anisotropic
groups of type $A_{n}$.

The concept of the commuting involution graph was introduced by Fischer \cite{F}, where he investigated three-transposition groups generated by conjugacy classes of involutions, which led eventually to the discovery of three new sporadic groups. An intensive study of commuting involution graphs was conducted by Bates, Bundy, Perkins, and Rowley. For instance, the commuting involution graphs for finite Coxeter groups, symmetric groups, and special linear groups have been investigated in \cite{BBPR03a, BBPR03b, BBPR}. See also more recent work on three-dimensional unitary groups \cite{Eve}.

The main theme of the present paper is to explore the commuting involution graphs of general and special linear groups over an arbitrary field. As mentioned above, this work was initiated in \cite{BBPR}, and in particular \cite{BBPR} provides an upper bound for the diameter of the commuting involution graphs for   special linear groups over a field of characteristic $2$ \cite[Theorem 1.3]{BBPR}. The present paper establishes the diameter over an arbitrary field by improving this upper bound.

Let $x$ be an involution (i.e., an element of order $2$) in $\GL_{n}(F)$. Then, the element $x$ is similar to
\begin{equation*}
I(n, k):=\begin{cases}
	\left(\begin{array}{c|c}
		I_{n-k}\\
		\hline  & -I_{k}
	\end{array}\right) \text{ if } \Ch(F)\neq 2, \\

	\left(\begin{array}{c|c}
		I_{n-k}\\
		 \hline		
	\begin{array}{c|c}
		& I_{k}\end{array} & I_{k} 
\end{array}\right) \text{ if } \Ch(F)= 2,
\end{cases}
\end{equation*}
for some $1\leq k\leq n$  in the first case and $1\leq k\leq n/2$ in the second case. The conjugacy class of $I(n, k)$ will be denoted by $X_{k}$. Since $X_n$ consists of a single element $-I_n$ if $\Ch(F)\neq 2$, we will not consider this trivial case.

In this paper, we completely determine the diameter of commuting involution graph $\Gamma(G,X_{k})$ for $G=\GL_{n}(F)$ and $\SL_{n}(F)$ over an arbitrary field $F$ as follows:

\begin{theorem}\label{mainthm}
Let $G=\GL_{n}(F)$ over an arbitrary field $F$ and let $X_{k}$ be the class of involutions for $1\leq k\leq n-1$  in case  $\Ch(F)\neq 2$ and $1\leq k\leq n/2$  in case  $\Ch(F)=2$. Then, for any $n\geq 3$ and any $X_k,$ the commuting involution graph $\Gamma(G,X_{k})$ is connected of diameter at most $4$. More precisely,
\begin{enumerate}
	\item[(i)] If $4k\leq n$ or $4(n-k)\leq n$, then $\diam\Gamma(G,X_{k})=2.$
	
	\item[(ii)] If $2k<n<4k$ or $2(n-k)<n<4(n-k)$ or $n=2k$ with even $k$, then $\diam\Gamma(G,X_{k})=3.$
	 
	  \item[(iii)] If $n=2k$ with odd $k$, then $\diam\Gamma(G,X_{k})=3$ for $ \Ch(F)\neq 2 $ and  $\diam\Gamma(G,X_{k})=4$ for $ \Ch(F)=2. $ 	
\end{enumerate}
\end{theorem}

We would like to point out that our diameter for $\GL_n(F)$ does not depend on the characteristic of the base field $F$ but only on a certain ratio between $n$ and $k$ except in case $n=2k$ with odd $k$. Moreover, as $\Gamma(\GL_{n}(F), X_{k})\simeq \Gamma(\SL_{n}(F), X_{k})$ together with the conjugacy class $X_{k}$ for even $1\leq k\leq n-1$ in case $\Ch(F)\neq 2$, we obtain the same statement for $G=\SL_{n}(F)$.  Hence, we provide the exact value of the diameter of the commuting graph for $\SL_{n}(F)$, which improves the main results on the upper bounds of the diameter over $\Ch(F)=2$ in \cite[Theorem 1.3]{BBPR}.


As an immediate application of Theorem \ref{mainthm}, we obtain the same result for certain projective special linear groups.
\begin{corollary}
Let $G=\PSL_{n}(F)$ over a finite field $F$ with $q$ elements and let $\bar{X}_{k}$ be the corresponding class of involutions for even $1\leq k\leq n-1$  in case  $\Ch(F)\neq 2$ and $1\leq k\leq n/2$  in case  $\Ch(F)=2$. Assume that $\gcd(n, q-1)$ is odd. Then, for any $n\geq 3$ and any $\bar{X}_k,$ the same result holds as in Theorem \ref{mainthm}
\end{corollary}
\begin{proof}
It suffices to show that $\Gamma(\SL_{n}(F), X_{k})\simeq \Gamma(\PSL_{n}(F), \bar{X}_{k})$ over a finite field $F$ with $q$ elements. Let $x, y\in X_{k}$ and  $Z(\SL_{n}(F))=\langle a \rangle$. Assume $\bar{x}=\bar{y}$ in $\bar{X}_{k}$, i.e., $x=a^{i}y$ for some $1\leq i\leq\gcd(n, q-1)$. Then, $a^{2i}=1$;  thus, by assumption we obtain $x=y$ in $X_{k}$ and $|X_{k}|=|\bar{X}_{k}|$. Moreover, if $\bar{x}\bar{y}=\bar{y}\bar{x}$, i.e., $(xy)^{2}=a^{i}I_{n}$ for some $1\leq i\leq\gcd(n, q-1)$, then $a^{2i}=1$; thus, $xy=yx$. Hence, they have the isomorphic graphs.
\end{proof}
Applying similar arguments, one can see that the diameter in Theorem \ref{mainthm} provides an upper bound of the diameter for certain involution classes of $\PSL_{n}(F)$ for even $n$ and some odd $q$.

Now we consider the set $X$ of all involutions in $\GL_{n}(F)$, i.e., $X=\bigcup X_{k}$. If $F$ is finite and $n\geq 4$, then a theorem of Brauer-Fowler \cite[Theorem (3D)]{BF} says $\diam\Gamma(\GL_{n}(F), X)\leq 3$ as mentioned before. For an arbitrary field $F$, a similar result \cite[Theorem 13]{AMRR} was obtained by Akbari, Mohammadian, Radjavi, and Raja: $\diam\Gamma(\GL_{n}(F), X)=3$ if $\Ch(F)\neq 2$ and $\diam\Gamma(\GL_{n}(F), X)\leq 4$ if $\Ch(F)=2$. By using the method developed in the proof of the main theorem, we prove 
\begin{corollary}
	Let $F$ be a field of characteristic $2$, $n\geq 3$, and $X$ the set of all involutions in $\GL_{n}(F)$. Then, $\diam\Gamma(\GL_{n}(F), X)=3$. 
\end{corollary}

Consequently, a theorem of Brauer-Fowler also holds for $\GL_{n}(F)$ over an arbitrary field $F$. We remark that there have been a number of recent works on the diameter of the commuting graphs of the general linear groups and matrix algebras over a field. Although we are dealing with a different set of vertices, it would be interesting to compare our result with the results in \cite{AMRR, DKKO, Shitov}.

As another application of our main theorem, we determine the structure of four dimensional linear group over a finite field (Proposition \ref{fourx1x3} and Proposition \ref{fourx2}). More precisely, for a fixed involution $t$ and a class $X_{k}$ in $\GL_{4}(F)$, we classify forms of involution $x\in X_{k}$ according to the distance $d(t,x)$ and provide size of $\Delta_{i}(t):=\{x\in X_{k}\mid d(t, x)=i\}$ for each $i$. We remark that for a three dimensional linear group, a similar result was previously obtained in \cite[Theorem 1.2]{BBPR}.

The present paper is organized as follows. In Section \ref{GL}, we first determine the distance from a fixed involution to an involution in triangular form (Proposition \ref{lowertriangular}). Using this result, we then present the proof of our main theorem in Theorem \ref{odd diam} and Theorem \ref{even diam} depending on characteristic of the ground field. At the end of this section, we determine the diameter of the graph on the set of all involutions over a field of characteristic $2$ (Corollary \ref{allinvolution}). In Section \ref{fdsec}, we present the detailed structure of involutions in $\GL_{4}(F)$ over a finite field (Proposition \ref{fourx1x3} and Proposition \ref{fourx2}). In the proof we apply our method developed in the previous section.

In this paper, we denote by $e_{ij}$ the matrix with a one in position $i, j$ and zeros elsewhere. For $a\in F$, we write $E_{ij}(a)$ for $I_{n}+ae_{ij}$. We shall use $\mat(m,n)$ to denote the set of all $m\times n$ matrices over a field $F$. For any elements $x, y$ in a group, we simply wirte $x\sim y$ if $xy=yx$.

\section{Commuting involutions in general linear groups}\label{GL}

\subsection{The diameter of triangular forms} For convenience, we shall write $t$ for $I(n, k)$. Then, by a direct computation we see that $\Delta_{1}(t)$ consists of partitioned matrices 
\begin{equation}\label{deltaone}
\left(\begin{array}{c|c}
A\\
\hline  & C
\end{array}\right)\in X_{k} \,\,\big(\text{resp. }\left(\begin{array}{c|c}
A'\\
\hline B & C
\end{array}\right)\in X_{k}\big) \,\,\text{ if } \Ch(F)\neq 2 \,\,(\text{resp. otherwise}),
\end{equation}
such that $A'=\left(\begin{array}{c|c}
P & Q\\
\hline  & C
\end{array}\right)$, $A^{2}=I_{n-k}$, $C^{2}=I_{k}$, $P^{2}=I_{n-2k}$, and $A'B=BC$.

Let $V$ be an $n$-dimensional vector space over $F$ with the natural action of $G=\GL_{n}(F)$. For an involution $x\in X_{k}$ we denote by $[V, t]$ the image of $x-1$, i.e., $[V,x]=\{xv-v\mid v\in V\}$, thus $\dim_{F}[V, x]=k$. Consider the conjugation action of $G$ on $[V, x]$ given by $g\cdot[V, x]=[V, gxg^{-1}]$. Note that if $\{v_{1},\ldots,v_{k}\}$ is a basis of $[V,t]$, then $\{gv_{1},\ldots,gv_{k}\}$ is a basis of $g\cdot[V, t]$. It follows from this that the stabilizer of $[V, t]$ is given by
\begin{equation}\label{stabilizeroft}
C_{G}([V,t])=\{\left(\begin{array}{c|c}
A\\
\hline B & C
\end{array}\right)\,|\,A\in \GL_{n-k}(F),B\in \mat(n-k,k),C\in \GL_{k}(F)\};
\end{equation}
see \cite[proof of Theorem 4.6]{BBPR}.

Let $W=[V, x]$ for an involution $x\in X_{k}$ and $m=\dim(W\cap [V,t])$. Then, we have $\max\{2k-n,0\}\leq m\leq k$. Moreover, it is easy to see that the stabilizer in (\ref{stabilizeroft}) acts transitively on the collection of such $k$-dimensional subspaces $W$ (see also \cite[proof of Theorem 4.6]{BBPR}). 

Let $J(n, k)=\left(\begin{array}{c|c}
	I_{n-k}\\
	\hline \begin{array}{c|c}
		I_{k}\end{array} & I_{k}
\end{array}\right)$ for any $1\leq k\leq n/2$. We shall frequently use the following involution in  $X_{k}$: 
\begin{equation}
t_{m}=\begin{cases}
\left(\begin{array}{c|c}
I(n-k,k-m)\\
\hline  & -I(k,k-m)
\end{array}\right)\text{ if }\Ch(F)\neq2\\
\left(\begin{array}{c|c|c}
I_{n-2k} & \\
\hline  & J(k,k-m)\\
\hline  & \begin{array}{c|c|c}
& &\\
\hline  & I_{2m-k} &\\
\hline  & &\\
\end{array} & J(k,k-m)
\end{array}\right)\text{ if }\Ch(F)=2\text{ and } 2m\geq k
\end{cases}\label{tm}
\end{equation}
such that $t=w_{m}t_{m}w_{m}$  and $t\sim t_m$ with $m=\dim([V, t_{m}]\cap [V,t])$, where
\begin{equation}
w_{m}=\left(\begin{array}{c|c|c|c}
I_{n-2k+m} &  & \\
\hline  &  & & I_{k-m}\\
\hline  &  &I_{m} \\
\hline  & I_{k-m} &  & 
\end{array}\right)
\Bigg(\text{resp. }\left(\begin{array}{c|c|c|c}
I_{n-2k+m} &  & \\
\hline  &  & I_{k-m}\\
\hline  & I_{k-m} & \\
\hline  &  &  & I_{m}
\end{array}\right) \Bigg)
\end{equation}
in the corresponding cases. Similarly, we define $t_m$ to be the transpose of  the second involution $t_{k-m}$ in (\ref{tm}) if $\Ch(F)=2$ and $2m<k$. In this case, we have $t^{T}=w_{k-m}t_{m}w_{k-m}$ and $t^{T}\sim t_{m}$.

\begin{lemma}\label{vxvteq}$(cf.\, $\cite[Lemma 4.1]{BBPR}$)$
Let $t\neq x\in X_{k}$ with $[V, t]=[V, x]$. Then, $x$ is of the form $x=\left(\begin{array}{c|c}
I_{n-k}\\
\hline  A & -I_{k}
\end{array}\right)$ for some $A\in \mat(k, n-k)$ if $\Ch(F)\neq 2$ and $x=\left(\begin{array}{c|c}
	I_{n-k}\\
	\hline  B & I_{k}
\end{array}\right)$ for some $B\in \mat(k, n-k)$ with $\rank(B)=k$ otherwise. In particular, $d(t, x)=1$ in the latter case.
\end{lemma}
\begin{proof}
Let $x=gtg^{-1}$ for some $g\in \GL_{n}(F)$. Then, by assumption we obtain $g\cdot [V, t]=[V, t]$, thus $g\in C_{G}([V,t])$. As $g$ is a block lower triangular as in (\ref{stabilizeroft}), the result immediately follows. The second statement follows from (\ref{deltaone}).\end{proof}

We first determine the distance from $t$ to an involution which is a block lower triangular.

\begin{proposition}\label{lowertriangular}
Let $x=\left(\begin{array}{c|c}
	A\\
	\hline B & C
	\end{array}\right)\in X_{k}$ with $A\in \GL_{n-k}(F),$ $B\in \mat(k, n-k)$, $C\in \GL_{k}(F)$. Then, $d(t,x)\leq 3$ if $\Ch(F)\neq 2$ and $n=2k$ with $k$ odd and $d(t,x)\leq 2$ otherwise.
\end{proposition}
\begin{proof}
First consider the case that $\Ch(F)=2$. Let $y= \left(\begin{array}{c|c}
J(n-k, k_{1}) & \\
\hline D & J(k, k_{2})
\end{array}\right)\in X_{k}$ for some $k_{1}, k_{2}\geq 0$ and $D\in \mat(k, n-k)$. We first show that $k_{1}\geq k_{2}$. Since $DJ(n-k, k_{1})=J(k, k_{2})D$, we see that $D=\left(\begin{array}{c|c}
D_{1} & \begin{array}{c|c} & \\ \hline D_{2}& \end{array}\\
\hline
D_{3} & D_{4}
\end{array}\right)$ for some $D_{1}\in \mat(k-k_{2}, k_{1})$, $D_{2}\in \mat(k-2k_{2}, n-k-2k_{1})$, $D_{3}\in \mat(k_{2}, k_{1})$, $D_{4}\in \mat(k_{2}, n-k-k_{1})$. As $y\in X_{k}$, we have $k_{1}\geq k_{2}$, thus we can find a block diagonal matrix $g\in C_{G}([V,t])$ such that $$y=g\left(\begin{array}{c|c}
\begin{array}{c|c} J(n-2k, k_{1}-k_{2}) & \\\hline & J(k, k_{2}) \end{array}  & \\ \hline
D' & J(k, k_{2})
\end{array}\right)g^{-1}=:gzg^{-1}$$ for some $D'\in \mat(k, n-k)$. Let $x=\left(\begin{array}{c|c}
A\\
\hline B & C
\end{array}\right)\in X_{k}$ for some $A\in \GL_{n-k}(F)$, , $C\in \GL_{k}(F)$. As $A^{2}=I_{n-k}$ and $C^{2}=I_{k}$, there exists a block diagonal matrix $h\in C_{G}([V,t])$ such that $x=hyh^{-1}=hgz(hg)^{-1}$. Since $z\in \Delta_{1}(t)$, it follows from Lemma \ref{vxvteq} that $d(t, x)\leq d(t, hgt(hg)^{-1})+d(hgt(hg)^{-1}, hgz(hg)^{-1})\leq 1+1=2$.

Consider the case that $\Ch(F)\neq 2$. Let $y'=\left(\begin{array}{c|c}
-I(n-k,n+m-2k)\\
\hline E & I(k,m)
\end{array}\right)
\in X_{k}$ for some $0\leq m\leq k$ and $E\neq 0$. Since $(y')^{2}=I_{n}$,
we see that the matrix $E$ is of the form $E=\left(\begin{array}{c|c}
E_{1}\\
\hline  & E_{2}
\end{array}\right)$ for some $E_{1}\in M(k-m,k-m)$ and $E_{2}\in M(m,n+m-2k)$. Let $P_{i}$, $Q_{i}$, $1\leq i\leq2$, be products of elementary
matrices such that 
\begin{equation}\label{PEQ}
P_{i}E_{i}Q_{i}=\left(\begin{array}{c|c}
I_{l_{i}} &\\
\hline & \\
\end{array}\right),\text{where }\,l_{i}=\rank(E_{i})
\end{equation}
and let $g'=\left(\begin{array}{c|c|c|c}
Q_{1} &  & \\
\hline  & Q_{2} & \\
\hline  &  & P_{1}^{-1}\\
\hline  &  &  & P_{2}^{-1}
\end{array}\right)$.  Consider a diagonal matrix $t'(\neq t)$ obtained by permuting
the diagonal entries of $t$, written as $t'=\left(\begin{array}{c|c|c|c}
R_{1} &  & \\
\hline  & R_{2} & \\
\hline  &  & S_{1}\\
\hline  &  &  & S_{2}
\end{array}\right)$ with $R_{1},S_{1}\in M(k-m,k-m)$, $S_{2}\in M(m,m),R_{2}\in M(n+m-2k,n+m-2k)$ such that   
\begin{equation}\label{SPEQ}
S_{i}(P_{i}E_{i}Q_{i})=(P_{i}E_{i}Q_{i})R_{i}.
\end{equation}
This is guaranteed by (\ref{PEQ}). Note that $n=2k$ with  $k=l_{1}+l_{2}$ odd if and only if there is no $t'\in X_{k}$. To see the forward direction, assume $t'\in X_{k}$. Then, the assumption and (\ref{PEQ}) implies $R_{i}=S_{i}$ for all $ i $, which is impossible. Similarly, one can easily check the reverse direction.

If $t'\in X_{k}$, then by (\ref{SPEQ}) we see that $y'$ commutes with $g't'(g')^{-1}\in \Delta_{1}(t)$. As $x=h'y'(h')^{-1}$ for some block diagonal matrix $h'\in C_{G}(t)$, we conclude that $x$ commutes with $(h'g')t'(h'g')^{-1}\in \Delta_{1}(t)$, i.e., $d(t, x)=2$. 

Otherwise, we see that our matrix 
\[
y'=\left(\begin{array}{c|c|c|c}
-I_{k_{1}}&  &  &  \\ \hline 
& I_{k_{2}} &  &  \\\hline 
E_{1}&  & I_{k_{1}} &  \\ \hline
& E_{2} &  & -I_{k_{2}} 
\end{array}  \right),
\]
where $k_{1}=k-m, k_{2}=m$, is a composition of 2 submatrices 
\[
y'_{1}=\left( \begin{array}{c|c}
-I_{k_{1}}&  \\ \hline
E_{1}&I_{k_{1}} 
\end{array} \right),\, y'_{2}=
\left( \begin{array}{c|c}
I_{k_{2}}&  \\ \hline
E_{2}&-I_{k_{2}} 
\end{array} \right).
\]
Therefore, it is enough to check that $d(t_{i}, y'_{i})\leq 3$ for each $1\leq i\leq 2$, where $t_{i}=I(2k_{i},k_{i})$. We shall consider only the case $i=1$ because the other case is similar. Thus, we may assume that $k=k_{1}$ is odd, $y'=y_{1}'$, and $t=t_{1}$. Write $y'=u\left(\begin{array}{c|c}
-I_{k}\\
\hline I_{k} & I_{k}
\end{array}\right)u^{-1}=:uz'u^{-1}$ for some block diagonal matrix $u\in C_{G}(t)$. As $z'$ commutes with $x':=\left(\begin{array}{c|c}
I(k, \lceil k/2\rceil)\\
\hline  U & I(k, [k/2])
\end{array}\right)$, where $U=(I(k, [k/2])-I(k, \lceil k/2\rceil))/2$, $x$ commutes with $(h'u)x'(h'u)^{-1}$. Since $\rank(U)\neq k$, it follows from the previous result that $d(t, (h'u)x'(h'u)^{-1})=2$, thus $d(t, x)\leq 3$. Indeed, $d(t, x)=3$ as $z'$ does not commute with any element in $\Delta_{1}(t)$. \end{proof}

\begin{remark}\label{uppertwo} 
By the same argument as in the proof of Proposition \ref{lowertriangular}, we see that for any $x\in X_{k}$ of the form $\left(\begin{array}{c|c}
	A & B\\
	\hline  & C
	\end{array}\right),$ $A\in \GL_{n-k}(F),$ $C\in \GL_{k}(F)$, $B\neq0$, $d(t, x)\leq3$ in case $\Ch(F)\neq 2$. More precisely, $d(t, x)\leq 3$ if $n=2k$ with $k$ odd and $d(t, x)\leq 2$ otherwise.
\end{remark}

\subsection{The diameter in characteristic different from $2$} In \cite[Theorem 10]{AMRR}, it was shown that for any $k=[n/2]$ $\diam \Gamma(G,X_{k})\geq 3$. We first present a general lower bound for the diameter in case $\Ch(F)\neq 2$. As $\Gamma(G, X_{k})\simeq \Gamma(G, X_{n-k})$ by the map $x\mapsto -x$, we shall only consider the case where $n\geq 2k$. 

\begin{proposition}\label{odd diam lower}
Let $F$ be a field of characteristic different from $2$ and $n\geq 3$. Then, for any integer $k$ with $n<4k<3n$, $\diam \Gamma(G,X_{k})\geq 3$.
\end{proposition}
\begin{proof}
It suffices to consider the case $2k<n<4k$. We shall find $x\in X_{k}$ such that it does not commute with any involution in $\Delta_{1}(t)$. Let $A=\left( \begin{array}{c|c}
&\\\hline I_{k-1}&
\end{array}\right)-2I_{k}\in \GL_{k}(F)$. Then, obviously, the only involution commuting with $A$ is $-I_{k}$. We divide the proof into the following two cases.

Assume $2k<n<3k$. Consider the following involution in $X_{k}$
\[x=g\left(\begin{array}{c|c|c}
I_{n-2k} & \\
\hline  & -I_{k} & I_{k}\\
\hline  &  & I_{k}
\end{array}\right)g^{-1}\, \text{ with } g=\left(\begin{array}{c|c|c}
I_{n-2k} & \\
\hline  & I_{k}\\
\hline  & A & I_{k}
\end{array}\right).\]
If $d(t,x)=2$, then by (\ref{deltaone}) there exists
\[y=g\left(\begin{array}{c|c|c}
B & \\
\hline C & D\\
\hline 2C &  & D
\end{array}\right)g^{-1}\in \Delta_{1}(t)
\]
such that $B^{2}=I_{n-2k}$, $D^{2}=I_{k}$, and $ AD=DA$. Hence, by assumption we have $D=I_{k}$, which contradicts $y\in X_{k}$. Therefore, $d(t,x)\geq 3$.

Now we assume that $3k\leq n<4k$. Consider an involution in $X_{k}$
\[
x=g\left(\begin{array}{c|c|c}
I_{n-2k} & \\
\hline \begin{array}{c|c} &I_{k}\end{array}  & -I_{k} & I_{k}\\
\hline  &  & I_{k}
\end{array}\right)g^{-1} \, \text{ with } g=\left(\begin{array}{c|c|c}
I_{n-2k} & \\
\hline \begin{array}{c|c} &-I_{k}\end{array} & I_{k}\\
\hline \begin{array}{c|c} &\,\,\,\;I_{k}\end{array} & A & I_{k}
\end{array}\right)
\]
Similarly, if $d(t, x)=2$, then by (\ref{deltaone}) there exists
\[y=g\left(\begin{array}{c|c|c}
B & \\
\hline C & I_{k}\\
\hline D &  & I_{k}
\end{array}\right)g^{-1}\in \Delta_{1}(t)\, \text{ with } B^{2}=I_{n-2k}\]
such that $CB+C=DB+D=0$,
\begin{equation}\label{prop2.4 eqn}
 C=\frac{1}{2}\left(\begin{array}{c|c}
B_{3}+D_{1} & B_{4}+D_{2}-I_{k}\end{array}\right), \text{ and } \left(\begin{array}{c|c} &I_{k}\end{array}\right)(B-I_{n-2k})+A\big(C+\left(\begin{array}{c|c} &I_{k}\end{array}\right)\big)+D=0,
\end{equation}
where $B=\left(\begin{array}{c|c}
B_{1}&B_{2}\\\hline B_{3}&B_{4}
\end{array} \right)$ and $D=\left(\begin{array}{c|c}D_{1} & D_{2}\end{array}\right)$ with $B_{4}, D_{2}\in \mat(k,k)$.

Write $B=PI(n-2k,k)P^{-1}$ for some $P=\left(\begin{array}{c|c}
P_{1} & P_{2}\\
\hline P_{3} & P_{4}
\end{array}\right)\in GL_{n-2k}(F)$ with $P_{4}\in\mat(k,k)$. Then, it follows by (\ref{prop2.4 eqn}) that $P_{3}=0$ and
\[\left(\begin{array}{c|c}
\, & -2I_{k}\end{array}\right)=\left(\begin{array}{c|c}
\frac{1}{2}A(B_{3}+D_{1})-D_{1} & -\frac{1}{2}AD_{2}-D_{2}\end{array}\right).\]
Therefore, we obtain $4I_{k}=(A+2I_{k})D_{2}$, which is impossible as $A+2I_{k}$ is singular. Hence, $d(t,x)\geq 3$.\end{proof}

Applying Propositions \ref{lowertriangular} and \ref{odd diam lower}, we first determine the diameter in case $\Ch(F)\neq 2$.

\begin{theorem}\label{odd diam}
Let $F$ be a field of characteristic different from $2$ and $n\geq 3$. Then, for any integer $k$ with $n<4k<3n$, $\diam \Gamma(G,X_{k})=3$.
\end{theorem}
\begin{proof}
It is enough to prove that the lower bounds in Proposition \ref{odd diam lower} are tight. We first assume that $ n>2k $ or $ n=2k $ with $ k $ even. Let $x\in X_{k}$ with $\dim([V,t]\cap[V,x])=m$.
By transitivity, there is a block lower triangular
matrix $g$ such that $[V,x]=[V,gt_{m}g^{-1}]=[V, gw_{m}tw_{m}g^{-1}]$.
Therefore, we have
\[
w_{m}g^{-1}xgw_{m}=\left(\begin{array}{c|c}
I_{n-k}\\
\hline B & -I_{k}
\end{array}\right)
\]
for some $B\in M(k,n-k)$. Let $B=\left(\begin{array}{c|c}
B_{1} & B_{2}\\ \hline
B_{3} & B_{4}
\end{array}\right)$ with $B_{1}\in \mat(m,n-2k+m)$, $B_{4}\in \mat(k-m,k-m)$.
Then, the product $g^{-1}xg$ is decomposed as 
\[
h^{\prime}\left(\begin{array}{c|c|c|c}
I_{n-2k+m} &  & \\
\hline  & -I_{k-m} & &B_{4}\\
\hline  &  & -I_{m}\\
\hline  &  &  & I_{k-m}
\end{array}\right)h^{\prime-1}, \text{ where } h^{\prime}=\left(\begin{array}{c|c|c|c}
I_{n-2k+m} &  & \\
\hline B_{3}/2 & I_{k-m} & \\
\hline B_{1}/2 &  & I_{m}&B_{2}/2\\
\hline  &  &  & I_{k-m}
\end{array}\right).
\]
Let $y$ denote the middle matrix in the above decomposition and $h=gh'$. Then, $x=hyh^{-1}$. By Remark \ref{uppertwo}, there
exists $z\in\Delta_{1}(t)$ such that $yz=zy$. Hence, $hzh^{-1}$
commutes with $x$. As $hzh^{-1}$ is block lower triangular, it follows
from Proposition \ref{lowertriangular} that $d(hzh^{-1},t)\leq2$. Therefore,
$d(t,x)\leq3.$

Now we consider the case where $n=2k$ and $k$ odd. Note that a conjugation on $x$ by a block diagonal matrix in $C_{G}(t)$ does not affect on the distance from $t$, i.e.,  all $C_{G}(t)$-conjugacy class of $x$ has the same distance from $t$. Hence, after a suitable modification, we may assume that 
\begin{equation}\label{modx}
x=\left(\begin{array}{c|c}
I_{k}\\
\hline L & I_{k}
\end{array}\right)\left(\begin{array}{c|c|c|c}
I_{m} &  & \\
\hline  & -I_{k-m} &  & D\\
\hline  &  & -I_{m}\\
\hline  &  &  & I_{k-m}
\end{array}\right)\left(\begin{array}{c|c}
I_{k}\\
\hline -L & I_{k}
\end{array}\right),\, \text{ where } D=\left(\begin{array}{c|c}
I_{l}\\
\hline  & 0
\end{array}\right)
\end{equation}
for some $0\leq l\leq k-m$ and $L\in M(k, k)$. We shall continue to denote by $h$ and $y$ the first and the second matrices in (\ref{modx}), respectively. If $l=0$, then $x$ is block lower triangular, thus by Proposition \ref{lowertriangular}, we have $d(t,x)\leq 3$.

Assume that $m=0$ and both $L$ and $D$ are singular ($0<l<k$). Observe that for any $M=\left(\begin{array}{c|c}
M_{1}\\
\hline  & M_{2}
\end{array}\right)$ and $N=\left(\begin{array}{c|c}
M_{1}\\
\hline N_{2} & N_{3}
\end{array}\right)$ with $M_{1}\in \GL_{l}(F)$, $M_{2},N_{3}\in \GL_{k-l}(F)$, the matrix $x=hyh^{-1}$ in (\ref{modx}) is $C_{G}(t)$-conjugate to $\left(\begin{array}{c|c}
I_{k}\\
\hline L' & I_{k}
\end{array}\right)y\left(\begin{array}{c|c}
I_{k}\\
\hline -L' & I_{k}
\end{array}\right)$, where $L'=N^{-1}LM$, thus we may replace $L$ by $L'$. Write $L=\left(\begin{array}{c|c}
L_{1} & L_{2}\\
\hline L_{3} & L_{4}
\end{array}\right)$, $L_{1}\in \mat(l,l)$, $L_{4}\in \mat(k-l,k-l)$ and $L'=(L_{i}')$ in the same form. If $L_{2}\neq 0$, then we choose $M_{1}$ and $M_{2}$ such that the last column of $L_{2}'=M_{1}^{-1}L_{2}M_{2}$ is $(0,\ldots,0, 1)^{T}$. Let $v$ be the last column of $L_{4}M_{2}$. Set $N_{2}=\left(\begin{array}{c|c}
\, & v
\end{array}\right)$. Then, all entries of the last column of $L_{4}'=-N_{2}M_{1}^{-1}M_{2}+L_{4}M_{2}$ are zero except the $l$-th element. Choose a diagonal matrix $t'$ obtained by permuting the diagonal entries of $t$ whose $l$-th, $k$-th, and $(l+k)$-th entries are all $-1$. Then, we see that $y\sim t'$ and the $(2,1)$-block of $ht'h^{-1}$ is singular. Hence by Proposition \ref{lowertriangular}, we obtain $d(t,x)=d(t, ht'h^{-1})+d(ht'h^{-1},x)\leq 3$. Similarly, if $L_{2}=0$, then we can find $M$ and $N$ as above such that
$L'=\left(\begin{array}{c|c}
J_{1} & \\
\hline L_{3}' & L_{4}'\end{array}\right)$,
where $J_{1}$ is a rational canonical from of $L_{1}$ and $L_{4}'=\left(\begin{array}{c|c}&\\ \hline  &I_{l'}\end{array}\right)$ for some $0\leq l'\leq k-l$. If $J_{1}$ is singular, then the $(2,1)$-block of $ht'h^{-1}$ is singular, thus we may assume that $J_{1}$ is nonsingular. Moreover, by applying appropriate $M$ and $N$, we may assume that $L_{3}'=0$, thus $L'$ has a zero row, which implies that the $(2,1)$-block of $ht'h^{-1}$ is singular and $d(t,x)\leq 3$. We divide the remaining proof into the following three cases.

\framebox{{\it Case}: $m=0$ and $L\in \GL_{k}(F)$ or $m=0$, $D=I_{k}$, and $L\not\in \GL_{k}(F)$.}
We start with the first case. Write $L=QP^{-1}$ for some $P,Q\in \GL_{k}(F)$. Then, $x$ is $C_{G}(t)$-conjugate to 
\begin{equation}\label{mzerox}
\left(\begin{array}{c|c}
I_{k}\\
\hline I_{k} & I_{k}
\end{array}\right)\left(\begin{array}{c|c}
-I_{k} & J\\
\hline  & I_{k}
\end{array}\right)\left(\begin{array}{c|c}
I_{k}\\
\hline -I_{k} & I_{k}
\end{array}\right),\, \text{ where } J=\left(\begin{array}{c|c}
J_{1}\\
\hline  & J_{2}
\end{array}\right) 
\end{equation}
is a rational canonical form of $P^{-1}DQ$ with a singular part $J_{1}\in \mat(k_{1},k_{1})$ and a nonsingular part $J_{2}\in \GL_{k_{2}}(F)$.

Let $A, B\in \GL_{k_{2}}(F)$ such that $AJ_{2}B^{-1}=I_{k_{2}}$. Then, the matrix in (\ref{mzerox}) is $C_{G}(t)$-conjugate to a matrix 
\begin{equation}\label{mzeroxtwo}
\left(\begin{array}{c|c|c|c}
I_{k_{1}} &  & \\
\hline  & I_{k_{2}} & \\
\hline I_{k_{1}} &  & I_{k_{1}}\\
\hline  & J' &  & I_{k_{2}}
\end{array}\right)\left(\begin{array}{c|c|c|c}
-I_{k_{1}} &  & J_{1}\\
\hline  & -I_{k_{2}} &  & I_{k_{2}}\\
\hline  &  & I_{k_{1}}\\
\hline  &  &  & I_{k_{2}}
\end{array}\right)\left(\begin{array}{c|c|c|c}
I_{k_{1}} &  & \\
\hline  & I_{k_{2}} & \\
\hline -I_{k_{1}} &  & I_{k_{1}}\\
\hline  & -J' &  & I_{k_{2}}
\end{array}\right)
\end{equation}
where $J'$ is a rational canonical form of $BA^{-1}$. As in the proof of Proposition \ref{lowertriangular}, the matrix in (\ref{mzeroxtwo}) is a composition of following two submatrices
\begin{equation}
\label{xonextwo}
x_{1}=\left(\begin{array}{c|c}
I_{k_{1}}\\
\hline I_{k_{1}} & I_{k_{1}}
\end{array}\right)\left(\begin{array}{c|c}
-I_{k_{1}} & J_{1}\\
\hline  & I_{k_{1}}
\end{array}\right)\left(\begin{array}{c|c}
I_{k_{1}}\\
\hline -I_{k_{1}} & I_{k_{1}}
\end{array}\right),\, x_{2}=\left(\begin{array}{c|c}
I_{k_{2}}\\
\hline J' & I_{k_{2}}
\end{array}\right)\left(\begin{array}{c|c}
-I_{k_{2}} & I_{k_{2}}\\
\hline  & I_{k_{2}}
\end{array}\right)\left(\begin{array}{c|c}
I_{k_{2}}\\
\hline -J' & I_{k_{2}}
\end{array}\right),
\end{equation}
it is enough to show that $d(t_{i},x_{i})\leq 3$ for all $1\leq i\leq 2$, where $t_{i}=I(2k_{i}, k_{i})$. For the same reason, we may assume that $J_{1}$ is a singular companion matrix and $J'$ is one of the followings: a diagonal matrix, a nonsingular companion
matrix or a direct sum of an element in $F^{\times}$, and a nonsingular companion matrix. If $k_{i}$ is even for some $i$, then it follows by the previous case that $d(t_{i},x_{i})\leq 3$. Therefore, it suffices to consider the distance for each odd $k_{i}$.

Let $k_{i}=2r_{i}+1$ for $1\leq i\leq 2$. Consider the following involution in $X_{k_{1}}$
\[
y_{1}=\left(\begin{array}{c|c}
I_{k_{1}}\\
\hline I_{k_{1}} & I_{k_{1}}
\end{array}\right)\left(\begin{array}{c|c}
Y & Z\\
\hline  & W
\end{array}\right)\left(\begin{array}{c|c}
I_{k_{1}}\\
\hline -I_{k_{1}} & I_{k_{1}}
\end{array}\right),\, \text{ where } Z=\frac{J_{1}W-YJ_{1}}{2},
\]
\[
Y=\left(\begin{array}{c|c|c}
I(r_{1},r_{1}-1) & \\
\hline  & -1\\
\hline  & u & I_{r_{1}}
\end{array}\right), W=\left(\begin{array}{c|c|c}
I_{r_{1}} & \\
\hline v & -1\\
\hline  &  & -I_{r_{1}}
\end{array}\right),\, u=\left(\begin{array}{c}
-1\\
0\\
\vdots\\
0
\end{array}\right),\, v=\left(\begin{array}{cccc}
0 & \cdots & 0 & -1\end{array}\right).
\]
Using the fact that the first row of $J_{1}$ is zero, one can easily check that
\[
x_{1}\sim y_{1}\sim \left(\begin{array}{c|c}
I(k_{1},r_{1})\\
\hline  & I(k_{1},r_{1}+1)
\end{array}\right)\in\Delta_{1}(t_{1}),
\]
thus $d(t_{1},x_{1})\leq 3$ for odd $k_{1}$.

Now we show that $d(t_{2},x_{2})\leq 3$ for odd $k_{2}$. Assume that $J'$ is a diagonal matrix. Then, we have 
\begin{equation}\label{pathforx2}
\left(\begin{array}{c|c}
-I_{k_{2}} & I_{k_{2}}\\
\hline  & I_{k_{2}}
\end{array}\right)\sim\left(\begin{array}{c|c}
I(k_{2},r_{2}) & -e_{(r_{2}+1)(r_{2}+1)}\\
\hline  & I(k_{2},r_{2}+1)
\end{array}\right)\sim\left(\begin{array}{c|c}
Y\\
\hline WJ'-J'Y & W
\end{array}\right), 
\end{equation}
where
\[
Y=I(k_{2},r_{2}),\, W=\left(\begin{array}{c|c}
-1\\\hline
& I(2r_{2},r_{2})
\end{array}\right).
\] 
Hence, after the conjugation by the first matrix in the decomposition of $x_{2}$ we obtain $d(t_{2},x_{2})\leq 3$. Similarly, if $J'$ is either a nonsingular companion
matrix or a direct sum of an element in $F^{\times}$ and a nonsingular companion matrix, then the same commutativity (\ref{pathforx2}) holds if we replace $Y$ and $W$ by
\begin{equation}
\label{YandWanother}
Y=\left(\begin{array}{c|c|c}
I(r_{2},r_{2}-1) & \\
\hline  & -1 & v\\
\hline  &  & I_{r_{2}}
\end{array}\right),\, W=\left(\begin{array}{c|c|c}
I_{r_{2}} & \\
\hline  & -1\\
\hline  &  & -I_{r_{2}}
\end{array}\right)
\end{equation}
where $v=\left(\begin{array}{ccc}
v_{1} & \cdots & v_{r_{2}}\end{array}\right)$ with  $v_{i}=-(J')_{r_{2}+1,r_{2}+1+i}$ for all $1\leq i\leq r_{2}$. Hence, $d(t_{2},x_{2})\leq 3$. For the second case, we may replace the matrix $L$ in (\ref{modx}) by its rational canonical form. Then, the same proof above works with $Y$ and $W$ in (\ref{pathforx2}) and (\ref{YandWanother}), replacing $k_{2}$ by $k$.

\framebox{{\it Case}: $k>3\text{ and } m>0$.} Write
$L=\left(\begin{array}{c|c}
L_{1} & L_{2}\\
\hline L_{3} & L_{4}
\end{array}\right)$ with $L_{1}\in \mat(m,m), L_{4}\in \mat(k-m,k-m)$. Find $M, N\in \GL_{m}(F)$ and a lower triangular matrix $P\in \GL_{k-m}(F)$ such that $M^{-1}L_{1}N$ is a diagonal matrix and the last column of $P^{-1}L_{3}N$ has at most one nonzero element in the $i_{0}$-th row. Then, we see that $x$ is $C_{G}(t)$-conjugate to
\begin{equation}\label{mod2x}
\left(\begin{array}{c|c}
I_{k}\\
\hline L' & I_{k}
\end{array}\right)\left(\begin{array}{c|c|c|c}
I_{m} &  & \\
\hline  & -I_{k-m} &  & D'\\
\hline  &  & -I_{m}\\
\hline  &  &  & I_{k-m}
\end{array}\right)\left(\begin{array}{c|c}
I_{k}\\
\hline -L' & I_{k}
\end{array}\right)
,
\end{equation}
where $P'\in \GL_{l}(F)$ is the (1,1)-block of $P$,
\[
L'=\left(\begin{array}{c|c}
M\\
\hline  & P
\end{array}\right)^{-1}L\left(\begin{array}{c|c}
N\\
\hline  & I_{k-m}
\end{array}\right)
\,\text{ and } D'=DP=\left(\begin{array}{c|c}
P'\\
\hline  & \,
\end{array}\right).\]
Hence, we may replace $x$ by the matrix in (\ref{mod2x}). We shall denote by $h'$ and $y'$ the first and the second matrices in (\ref{mod2x}), respectively.

We choose diagonal matrices $Y_{i}\in \GL_{m}(F), W_{i}\in \GL_{k-m}(F)$ for $1\leq i\leq 2$ such that the last diagonal entries of $Y_{i}$ and the $i_{0}$-th diagonal entry of $W_{2}$ are all $-1$, $DW_{2}=W_{1}D$, and 
\[
z:=\left(\begin{array}{c|c|c|c}
Y_{1} &  & \\
\hline  & W_{1}' & \\
\hline  &  & Y_{2}\\
\hline  &  &  & W_{2}
\end{array}\right)\in X_{k}\,\text{ with } W_{1}'=\left(\begin{array}{c|c}
P'\\
\hline  & I_{k-m-l}
\end{array}\right)W_{1}\left(\begin{array}{c|c}
P'\\
\hline  & I_{k-m-l}
\end{array}\right)^{-1}
\]
Then, we have $z\sim y'$ and the $m$-th column of (2,1)-block of $h'zh'^{-1}$ is zero, thus by Proposition \ref{lowertriangular} $d(t,h'zh'^{-1})\leq 2$, so $d(t,x)\leq3$.

\framebox{{\it Case}: $k=3 \text{ and }m>0$}
Assume that $m=2$. Then, by the same argument as in the previous case, we may assume that
\[
x=\left(\begin{array}{c|c}
I_{3}\\
\hline L & I_{3}
\end{array}\right)\left(\begin{array}{c|c|c|c}
I_{2} &  & \\
\hline  & -1 &  & 1\\
\hline  &  & -I_{2}\\
\hline  &  &  & 1
\end{array}\right)\left(\begin{array}{c|c}
I_{k}\\
\hline -L & I_{k}
\end{array}\right)=:hyh^{-1}, \text{ with } L=\left(\begin{array}{cc|c}
* & * & *\\
& * & *\\
\hline * &  & *
\end{array}\right)\text{ or }\left(\begin{array}{cc|c}
* &  & *\\
*& * & *\\
\hline  & * & *
\end{array}\right).
\]
Then, $y$ commutes with the diagonal matrix $z$ with entries $(-1,1,-1,1,1,-1)$ and $(1,-1,-1,1,1,-1)$. As the last row of (2,1)-block of $hzh^{-1}$ is zero, it follows from Proposition \ref{lowertriangular} that $d(t,x)\leq3$.

Now we assume that $m=1$. First of all, if either the first row or the second row in the $(2,2)$-block of $L$ (i.e., $L_{4}\in \mat(2,2)$) is zero, then the middle matrix $y$ in (\ref{modx}) commutes with the diagonal matrices $z_{1}=\diag(-1,-1,1,1,-1,1)$ or $z_{2}=\diag(-1,1,-1,1,1,-1)$. Moreover, the $(2,1)$-block of  $hz_{i}h^{-1}$ is singular for each $1\leq i\leq 2$. Therefore, by Proposition \ref{lowertriangular} we conclude that $d(t,x)\leq3$. Hence, after a suitable modification, we may assume that $L_{4}$ is invertible.

If $D$ is invertible (i.e., $D=I_{2}$) then, we have  
\[
x=\left(\begin{array}{c|c}
I_{3}\\
\hline L & I_{3}
\end{array}\right)\left(\begin{array}{c|c|c|c}
1 &  & \\
\hline  & -I_{2} &  & I_{2}\\
\hline  &  & -1\\
\hline  &  &  & I_{2}
\end{array}\right)\left(\begin{array}{c|c}
I_{3}\\
\hline -L & I_{3}
\end{array}\right)=:hyh^{-1}, \text{ where } L=\left(\begin{array}{c|cc}
i & a & b\\
\hline c &  & d\\
e & 1 & f
\end{array}\right)
\]
with $i=0$ or $1$ and $d\neq 0, a, b, c, e, f\in F$ (if $L_{4}$ is diagonalizable, then by the same argument as in the case $m=2$ we obtain $d(t,x)\leq 3$). One can easily verify that
\[
y\sim \left(\begin{array}{c|c}
Y &\\ \hline
Z& W
\end{array}\right)=:z,
\]
where
\begin{align}Y=\left(\begin{array}{c|c|c}1 &  & \\ \hline
		& 1 & \\ \hline
		q & p & -1\end{array}\right),\, W=\left(\begin{array}{c|c|c}-1 & r & \\ \hline
		& 1 & \\ \hline
		& p & -1\end{array}\right), Z=\left(\begin{array}{c|c|c} & -2r & \\ \hline
		&  & \\ \hline
		2q&  & \end{array}\right)\end{align}

\[\text{ with } (p, q, r)=\begin{cases} (2e/c, 0, 0) & \text{ if } i=0, c\neq 0,\\
(-2a/b, 0, 0) & \text{ if } i=0, c=0, b\neq0,\\
(2e/c, 0, 2/c) & \text{ if } i=1, c\neq 0,\\
(-2a/b, -2/b, 0) & \text{ if } i=1, c=0, b\neq 0,\\
((2a-4)/d, 2/d, 2) & \text{ if } i=1, c=b=0.
\end{cases}\]
If $ i=c=b=0, $ then take $ z=\diag(-1,1,-1,1,1,-1) $. As the $(2,1)$-block of $hzh^{-1}$ is singular, it follows from Proposition \ref{lowertriangular} that $d(t,x)\leq3$.

Finally, if $D$ is singular (i.e., $D=e_{11}$) and $L_{4}$ is nonsingular, then $x$ is $C_{G}(t)$-conjugate to  
\begin{equation}\label{modxkthree}
x'=\left(\begin{array}{c|c}
I_{3}\\
\hline L' & I_{3}
\end{array}\right)\left(\begin{array}{c|c|c|c}
1 &  & \\
\hline  & -I_{2} &  & D'\\
\hline  &  & -1\\
\hline  &  &  & I_{2}
\end{array}\right)\left(\begin{array}{c|c}
I_{3}\\
\hline -L' & I_{3}
\end{array}\right),\, \text{ where }\, L'=\left(\begin{array}{c|cc}
i & a & b\\
\hline c & 1\\
d &  & 1
\end{array}\right),D'=\left(\begin{array}{cc}
\\
1 & e
\end{array}\right)
\end{equation}
with $0\leq i\leq 1$ and $a, b, c, d, e\in F$. We denote by $h'$ and $y'$ the first and the second matrices in (\ref{modxkthree}), respectively. Consider an involution $z'$ in $X_{3}$
\begin{equation}\label{diagonalinvol}
z'=\begin{cases} \diag(1,1,-1,1,-1,-1) & \text{ if } d=0,\\
	\diag(1,-1,-1,1,-1,1) & \text{ if } c=e=0,\\
	\diag(-1,-1,1,1,-1,1)+(2/e)e_{65} & \text{ if } c=0, e\neq 0.\end{cases} 
\end{equation}
Then, we see that $y'\sim z'$ and the $(2,1)$-block of $h'z'h'^{-1}$ is singular, thus $d(t,x)\leq 3$ in corresponding cases. If $c, d\in F^{\times}$, then we replace $x'$ in  (\ref{modxkthree}) by the following $C_{G}(t)$-conjugate involution $x''$
\[
\left(\begin{array}{c|c}
I_{3}\\
\hline L'' & I_{3}
\end{array}\right)\left(\begin{array}{c|c|c|c}
1 &  & \\
\hline  & -I_{2} &  & D''\\
\hline  &  & -1\\
\hline  &  &  & I_{2}
\end{array}\right)\left(\begin{array}{c|c}
I_{3}\\
\hline -L'' & I_{3}
\end{array}\right),\text{ where } L''=\left(\begin{array}{c|cc}
i & a & b\\
\hline c & 1\\
& -d/c & 1
\end{array}\right),D''=\left(\begin{array}{c|c}
& \\ \hline
1+de/c & e
\end{array}\right).
\]
Applying the same argument together with $z'=\diag(1,1,-1,1,-1,-1)$, we conclude that $d(t,x)\leq 3$.\end{proof}

In the following corollary, we further reduce the upper bound for $\diam \Gamma(G,X_{k})$ over a field $F$ of characteristic different $2$. The same result in case $\Ch(F)=2$ was obtained in \cite[Lemma 4.3]{BBPR}.

\begin{corollary}\label{odd diam2}Let $F$ be a field of characteristic different $2$ and $n\geq 3$. Then, $\diam \Gamma(G,X_{k})=2$ if $n\geq 4k$ or $ n\geq 4(n-k)$.
\end{corollary}
\begin{proof}
Let $x=hyh^{-1}$ as in the proof of Theorem \ref{odd diam}. Write $h=u\left(\begin{array}{c|c}
I_{n-k}\\
\hline \begin{array}{c|c}
Q_{1}& Q_{2}\end{array} & I_{k}
\end{array}\right)=:uv$ for some block diagonal matrix $u\in C_{G}(t)$, $Q_{1}\in M(k, n-2k+m)$, and $Q_{2}\in M(k, k-m)$. Hence, $x=uvy(uv)^{-1}$.

We find $z\in \Delta_{1}(t)$ such that $yz=zy$ and $vzv^{-1}\in \Delta_{1}(t)$. By assumption that $n-3k+m\geq k$, we can find a matrix $P\in \GL_{n-2k+m}(F)$
such that the first $k$ columns of $Q_{1}P$ are all zero. Take 
\[z=\left(\begin{array}{c|c|c}
	C &  & \\
	\hline  & I_{k-m} & \\
	\hline  &  & I_{k}
\end{array}\right)
\]
where $C=-PI(n-2k+m, n-3k+m)P^{-1}$. Obviously, $z$ commutes with $y$. Moreover, by a direct computation we obtain $Q_{1}C=Q_{1}$, thus $vzv^{-1}\in \Delta_{1}(t)$. Therefore, $uvz(uv)^{-1}\in \Delta_{1}(t)$ and $d(t, x)=d(t, uvz(uv)^{-1})+d(uvy(uv)^{-1}, x)=2$.\end{proof}

\subsection{The diameter in characteristic $2$}

We first provide an involution which has diameter at least $4$.

\begin{lemma}\label{k odd n 2k even}
Let $F$ be a field of characteristic $2$ and $n\geq 3$. Then, $\diam \Gamma(G,X_{k})\geq 4$ if $n=2k$ with $k$ odd.
\end{lemma}

\begin{proof}
We show that $d(t,t^{T})\geq 4$. By (\ref{deltaone}), it is enough to show that $d(t, t^{T})\neq 3$. Assume that $d(t, t^{T})=3$, i.e., there exist commuting involutions $x\in\Delta_{1}(t)$ and $y\in\Delta_{1}(t^{T})$. As $x, y\in X_{k}$ and $xy=yx$, we may assume that 
\[x=\left(\begin{array}{c|c}
J(k, r)\\
\hline B & J(k, r)
\end{array}\right)
\text{ and } 
 y=\left(\begin{array}{c|c}
P & Q\\
\hline  & P
\end{array}\right)
\]
for some involution $P\in \GL_{k}(F)$ and $1\leq r\leq [k/2]$. Note that the centralizer $C_{M(k, k)}(J(k,r))$ consists of 
\[\left(\begin{array}{c|c|c}
	B_{1}  & & \\
	\hline B_{2} & B_{3} &\\
	\hline B_{4} & B_{5} & B_{1}
	\end{array}\right) \text{ with } B_{1}\in M(r,r), B_{3}\in \GL_{k-2r}(F).
	\]
Thus, the matrix $Q=(Q_{i})$, $1\leq i\leq 5$, is of the form as above. Choose $M, N\in C_{G}(J(k,r))$ such that $R=M^{-1}BN$, where 
\[
R=\left(\begin{array}{c|c|c}
I' & \\
\hline  & I_{k-2r}\\
\hline S &  & I'
\end{array}\right) \text{ with } S=\left(\begin{array}{c|c}
\\
\hline  & S'
\end{array}\right),\, I'=\left(\begin{array}{c|c}
I_{l} & \\
\hline  & 
\end{array}\right)
\]
for some $S'\in M(r-l,r-l)$ and $0\leq l\leq r$, thus $x=\left(\begin{array}{c|c}
N\\
\hline  & M
\end{array}\right)\left(\begin{array}{c|c}
J(k,r)\\
\hline R & J(k,r)
\end{array}\right)\left(\begin{array}{c|c}
N\\
\hline  & M
\end{array}\right)^{-1}$. For simplicity, we shall assume that $M=N=I_{k}$. Indeed, the same argument below works for $y=\left(\begin{array}{c|c}
	N^{-1}PN & N^{-1}QM\\
	\hline  & M^{-1}PM
\end{array}\right)$.

Observe that $x\sim y$ implies that
\begin{equation}\label{equivalentequation}
J(k,r)P+PJ(k,r)=RQ=QR\, \text{ and }  PR=RP.
\end{equation}
Write $P=(P_{i})$ with $P_{1}, P_{9}\in M(r,r)$, $P_{5}\in M(k-2r, k-2r)$, $1\leq i\leq 9$ as below. Then, by the first equation in (\ref{equivalentequation}), we obtain that
\begin{equation}
\label{firsteq1}
P_{3}=\left(\begin{array}{c|c}
P_{3}'\\
\hline  & 
\end{array}\right), Q_{1}=\left(\begin{array}{c|c}
P_{3}'\\
\hline  & Q_{1}'
\end{array}\right), Q_{2}=P_{6}=\left(\begin{array}{c|c} P_{6}' & \end{array}\right),\, Q_{4}=\left(\begin{array}{c|c}
Q_{4}'\\
\hline  & Q_{4}''
\end{array}\right), Q_{5}=P_{2}=\left(\begin{array}{c} P_{2}'\\ \hline \, \end{array}\right),
\end{equation}
$Q_{3}=0$, and $P_{1}+P_{9}=\left(\begin{array}{c|c}
Q_{4}'\\
\hline  & Q_{1}'S'=S'Q_{1}'
\end{array}\right)$ for some $P_{3}', Q_{4}'\in M(l,l)$, $P_{6}'\in M(k-2r, l)$, $P_{2}'\in M(l, k-2r)$. Similarly, by the second equation in (\ref{equivalentequation}) we have
\begin{equation}\label{secondeq2}
P_{1}=\left(\begin{array}{c|c}
P_{1}'\\
\hline  & P_{1}'' 
\end{array}\right),\, P_{4}=\left(\begin{array}{c|c} P_{4}' & \end{array}\right),\, P_{7}=\left(\begin{array}{c|c}
P_{7}'\\
\hline  & P_{7}'' 
\end{array}\right),\, P_{8}=\left(\begin{array}{c} P_{8}'\\ \hline \, \end{array}\right),\, P_{9}=\left(\begin{array}{c|c}
P_{9}'\\
\hline  & P_{9}'' 
\end{array}\right),
\end{equation}
where $P_{1}', P_{7}', P_{9}'\in M(l,l)$, $P_{4}'\in M(k-2r, l)$, and $P_{8}'\in M(l, k-2r)$ such that $(P_{1}'')^{2}=(P_{9}'')^{2}=I_{r-l}$.

Using (\ref{firsteq1}) and (\ref{secondeq2}) together with elementary operations we see that $y$ is equivalent to $y'=\left(\begin{array}{c|c}
	P & Q'\\
	\hline  & P
\end{array}\right)$, i.e. $ \rank y=\rank y' $ where
\[P=\left(\begin{array}{c|c||c||c|c} P_{1}' & & P_{2}' & P_{3}' &\\ \hline
 & P_{1}'' & & &\\ \hline\hline
 P_{4}' & & P_{5} & P_{6}' & \\ \hline\hline
 P_{7}' & & P_{8}' & P_{9}' & \\ \hline
 & P_{7}'' & & & P_{9}''\end{array}\right) \text{ and } Q'=\left(\begin{array}{c|c||c||c|c}  & &  &  &\\ \hline
 & Q_{1}' & & &\\ \hline\hline
  & &  &  & \\ \hline\hline
  & &  &  & \\ \hline
 & Q_{4}'+Q_{1}'S' & & & Q_{1}'\end{array}\right). \]

Let $P'$ be the matrix obtained by deleting the second and fifth rows and columns of $P$. Then, we have $(P')^{2}=I_{k+2l-2r}$ as $(P)^{2}=I_{k}$, thus $\rank(P'+I_{k+2l-2r})\leq (k+2l-2r-1)/2$. Let $y''$ be the matrix obtained from $y'$ by deleting all rows and columns of $P$ and $Q'$ except the second and fifth rows and columns in each $P$ and $Q'$. Then, it follows by $(P)^{2}=I_{k}$ and $PQ=QP$ that $(y'')^{2}=I_{4r-4l}$. Hence, $\rank(y''+I_{4r-4l})\leq 2r-2l$. Therefore, $\rank(y+I_{n})=\rank(y'+I_{n})=2(\rank(P'+I_{k+2l-2r})+\rank(y''+I_{4r-4l}))\leq k-1$, which implies that $y\not\in X_{k}$, i.e., $d(t,x)\geq 4$.\end{proof}

In the remaining cases, we gives lower bounds of the diameter in $\Ch(F) =2$.

\begin{proposition}\label{even diam lower}
	Let $F$ be a field of characteristic $2$ and $n\geq 3$. Then, $\diam \Gamma(G,X_{k})\geq 4$ if $n=2k$ with $k$ odd and $\diam \Gamma(G,X_{k})\geq 3$ otherwise.
\end{proposition}
\begin{proof}
	By Lemma \ref{k odd n 2k even}, it remains to prove the case except that $n=2k$ with $k$ odd. We find an involution $x\in X_{k}$ which does not commute with any involution in $\Delta_{1}(t)$. Observe that any square matrix which commutes with $J(k, r)$ is of the form $\left(\begin{array}{c|c}
	\begin{array}{c|c} X & \end{array} & \, \\
	\hline Y & \begin{array}{c}  \\ \hline X \end{array}
	\end{array}\right)$
	with $X\in M(r,r)$. Choose an invertible matrix $B$ such that any conjugate of $B$ does not commute with $J(k,r)$ for all $1\leq r\leq [k/2]$. Such a matrix exists as we can take a diagonal matrix with distinct entries for $B$ if $F$ is infinite, and a matrix having an irreducible minimal polynomial of degree $k$ for $B$ otherwise. Consider an involution 
	\[
	x=\left(\begin{array}{c|c|c}
	I_{n-2k} & \\
	\hline A & I_{k} & B\\
	\hline  &  & I_{k}
	\end{array}\right)\in X_{k} \text{ with } A=\begin{cases} 0 & \text{ if } 2k\leq n<3k\\  (B\,|\,\,\,\,) & \text{ if } 3k\leq n<4k.\end{cases}
	\]
	
	
	We claim that $d(t,x)\geq 3$. If not, there exists 
	\[y=\left(\begin{array}{c|c|c}
	C & \\
	\hline  & D\\
	\hline E &  & D
	\end{array}\right)\in \Delta_{1}(t) \text{ with } D\in M(k,k),\, E=0 \text{ if }  2k\leq n<3k\]
	such that $AC+DA=BE$ and $DB=BD$. Hence, due to the choice of $B$, it suffices to show that $D$ is an involution. It is obvious that $y\not\in X_{k}$ if $2k\leq n<3k$ and $D=I_{k}$, thus $D$ is an involution. Now assume that $3k\leq n<4k$ and $D=I_{k}$. Then, $C$ is an involution as $B\in \GL_{k}(F)$ and $y\in X_{k}$. Hence, $C=PJ(n-2k,s)P^{-1}$ for some $1\leq s< k$ and $P\in\GL_{n-2k}(F)$. Write $P=\left(\begin{array}{c|c}
	P_{1} & P_{2}\\
	\hline P_{3} & P_{4}
	\end{array}\right)$ with $P_{1}\in M(k,k), P_{2}\in M(k, n-3k)$. Then, we obtain
	\[
	y=g\left(\begin{array}{c|c|c}
	J(n-2k,s) & \\
	\hline  & I_{k}\\
	\hline \left(\begin{array}{c|c}
	P_{1} & P_{2}\end{array}\right)(J(n-2k,s)+I_{n-2k}) &  & I_{k}
	\end{array}\right)g^{-1} \text{ with } g=\left(\begin{array}{c|c|c}
	P & \\
	\hline  & I_{k}\\
	\hline  &  & I_{k}
	\end{array}\right)
	\]
	and the rank of $ \left(\begin{array}{c|c}
	P_{1} & P_{2}\end{array}\right)(J(n-2k,s)+I_{n-2k}) $ is at most $s$. Since the nonzero elements are in the first $s$-columns, $\rank y=s<k$,  concluding that $y\notin X_{k}$. Thus, $D$ is an involution.\end{proof}

Now we establish tight upper bounds of the diameter in case $\Ch(F)=2$. For the case of $3k\leq n$, sharp upper bounds were obtained in \cite[Theorem 1.3]{BBPR}. In the following theorem, we consider the remaining case where $2k\leq n<3k$. In the proof, we shall use the involution in (\ref{tm}) and Proposition \ref{lowertriangular}.

\begin{theorem}\label{even diam}
Let $F$ be a field of characteristic $2$ and $n\geq 3$. Then, $\diam \Gamma(G,X_{k})=4$ if $n=2k$ with $k$ odd and $\diam \Gamma(G,X_{k})=3$ otherwise.
\end{theorem}
\begin{proof}
By Proposition \ref{even diam lower} it suffices to show that the lower bound is sharp. Let $x\in X_{k}$, $\dim([V,t]\cap[V,x])=m$, and $s=3k-n$. For the case of $3k\leq n$, we refer to the proof of \cite[Theorem 1.3]{BBPR}. So we may assume that $2k\leq n<3k$. First, we consider the case where $2m\geq k$. Then, by transitivity we have $[V,x]=[V,gt_{m}g^{-1}]$ for some $g\in C_{G}([V,t])$. Hence, by Lemma \ref{vxvteq}, we obtain $d(x, gt_{m}g^{-1})=d(t, gtg^{-1})=1$. As $t_m\sim t$, we conclude that $d(t, x)\leq d(t,gtg^{-1})+d(gtg^{-1},gt_{m}g^{-1})+d(gt_{m}g^{-1},x)\leq3$.

Now we assume that $2m<k$. Again by transitivity, there exists a block lower triangular
matrix $g$ such that $[V,x]=[V,gt_{m}g^{-1}]=[V, g(w_{k-m}t^{T}w_{k-m})g^{-1}]=[V, gw_{k-m}w_{0}tw_{0}w_{k-m}g^{-1}]$.
Hence, we have
\begin{equation}\label{xeventwo}
x=gw_{k-m}w_{0}\left(\begin{array}{c|c}
I_{n-k}\\
\hline B & I_{k}
\end{array}\right)w_{0}w_{k-m}g^{-1}
\end{equation}
for some $B\in M(k,n-k)$ with $\rank(B)=k$. Let $y$ denote the middle matrix in the above decomposition. We divide the proof into the following three cases (For the second and third cases, we always exclude the first case).

\framebox{{\it Case}: $n=2k$ with $k$ odd.} Let $k=2r+1$. Write $w_{0}yw_{0}=ut^{T}u^{-1}$ for some block diagonal matrix $u\in C_{G}([V,t])$. If $m=0$, then by (\ref{xeventwo}) we have $x=ht^{T}h^{-1}$, where $h=gu$. In this case, one has the following commuting involutions in $X_{k}$
\begin{equation}\label{firstcaseeq}
t^{T}\sim \left(\begin{array}{c|c}
J(k,r)^{T} & e_{(r+1)(r+1)}\\
\hline  & J(k,r)^{T}
\end{array}\right)\sim \left(\begin{array}{c|c}
I(k,r)^{T}\\
\hline e_{1k} & I^{\prime}(k,r)^{T}
\end{array}\right)=:z
\end{equation}
where, $I^{\prime}(k,r)=\left(\begin{array}{c|c}
I_{r}\\
\hline \begin{array}{c}
I_{r}\\
\hline \\
\end{array} & I_{k-r}
\end{array}\right)$. Since $hzh^{-1}$ is block lower triangular, it follows from Proposition \ref{lowertriangular} that $d(t, hzh^{-1})\leq 2$, thus $d(t,x)\leq 4$. Similarly, if $m>0$, then instead of  (\ref{firstcaseeq}) we use 
\[t^{T}\sim \left(\begin{array}{c|c}
I(k, r) & e_{11}\\
\hline  & I(k,r)
\end{array}\right)=:z'.\]
As $w_{k-m}z'w_{k-m}$ is block lower triangular, it follows from Proposition \ref{lowertriangular} that $d(t, gw_{k-m}z'w_{k-m}g^{-1})\leq 2$, thus $d(t,x)\leq 4$.

\framebox{{\it Case}: $2k\leq n<3k$ with $m=0$.} It follows from (\ref{xeventwo}) that $x=gw_{0}yw_{0}g^{-1}$. Let $P\in \GL_{k}(F)$ and let $Q\in \GL_{n-k}(F)$ be a block lower triangular matrix of the form $Q=\left(\begin{array}{c|c}
Q_{1} & \\\hline
Q_{2} & Q_{3}
\end{array}\right)$, $Q_{1}\in M(n-2k, n-2k)$, $Q_{3}\in M(k, k)$ such that $PBQ=\left(\begin{array}{c|c}
D_{1} & D_{2}\end{array}\right)$, where 
\[D_{1}=\left(\begin{array}{c|c}
I_{l} & \\ \hline & \end{array}\right)\in M(k, n-2k) \text{ and } D_{2}=\left(\begin{array}{c|c}
 & \\ \hline & I_{k-l}\end{array}\right)\in M(k, k)\]
for some $0\leq l\leq n-2k$. Then, we obtain 
\begin{equation}\label{eventhmformx}
x=h\left(\begin{array}{c|c|c}
I_{n-2k} & \\
\hline D_{1} & I_{k} & D_{2}\\
\hline  &  & I_{k}
\end{array}\right)h^{-1} \text{ with } h=gw_{0}\left(\begin{array}{c|c}
Q\\
\hline  & P^{-1}
\end{array}\right)w_{0}\in C_{G}([V, t]).
\end{equation}
Let $y'$ denote the middle matrix in the above decomposition. First assume that either $l<n-2k$ or $l=n-2k$ with $s$ even. Consider an involution
\[z=\left(\begin{array}{c|c|c} 
I_{n-2k} & &\\ \hline
\begin{array}{c} I_{n-2k}\\ \hline \, \end{array}
& I(k, \frac{s}{2})^{T} & \\ \hline
& & I(k,\frac{s}{2})^{T}
\end{array}\right) \big(\text{resp.} \left(\begin{array}{c|c|c} 
I_{n-2k} & &\\ \hline
\begin{array}{c|c} I_{n-2k-1} & \\ \hline \, \end{array}
& I(k, \lceil\frac{s}{2}\rceil)^{T} & \\ \hline
& & I(k, \lceil\frac{s}{2}\rceil)^{T}
\end{array}\right)\big) \in \Delta_{1}(t)\]
if $s$ even (resp. otherwise). Then, by a direct computation we have $z\sim y'$. Hence, by Proposition \ref{lowertriangular} we conclude that $d(t, x)=d(t, hzh^{-1})+d(hzh^{-1}, x)\leq 3$. If $l=n-2k$ with $s$ odd, the same proof works if we replace $z$ by
\[z'=\left(\begin{array}{c|c|c} 
I_{n-2k} & &\\ \hline
\begin{array}{c} I_{n-2k}\\ \hline \, \end{array}
& I'(k, [\frac{s}{2}]) & \\ \hline
& & I(k, \lceil\frac{s}{2}\rceil)^{T}
\end{array}\right)\in \Delta_{1}(t), \text{ where } I'(k, [\frac{s}{2}])=\left(\begin{array}{c|c|c} I_{k-\lceil\frac{s}{2}\rceil} & & \begin{array}{c} \, \\ \hline I_{[\frac{s}{2}]} \end{array} \\ \hline & 1 & \\ \hline & & I_{[\frac{s}{2}]}\end{array}\right).
\]

\framebox{{\it Case}: $2k\leq n<3k$ with $m>0$.} First assume that $s\leq 2m$. Then, $z:=w_{k-m}w_{0}J(n, k)w_{0}w_{k-m}$ is block lower triangular. Hence, by Proposition \ref{lowertriangular} we have $d(t,gzg^{-1})\leq 2$. As $J(n, k)\sim y$, we obtain $d(t, x)\leq 3$. From now on we assume that $s>2m$.

Let $P\in \GL_{k}(F)$ and let $Q\in \GL_{n-k}(F)$ be  block lower triangular matrices of the forms
\begin{equation}\label{formofpq}
P=\left(\begin{array}{c|c}
P_{1} & \\\hline
P_{2} & P_{3}
\end{array}\right) \text{ and } Q=\left(\begin{array}{c|c}
Q_{1} & \\\hline
Q_{2} & \begin{array}{c|c}
Q_{3} & \\\hline
Q_{4} & Q_{5}
\end{array}
\end{array}\right), P_{1}\in \GL_{k-m}(F), Q_{1}\in \GL_{n-2k}(F), Q_{3}\in \GL_{m}(F)
\end{equation}
such that $PBQ=\left(\begin{array}{c} \begin{array}{c|c}  \begin{array}{c|c} D_{1} & D_{2} \\ \hline &  \end{array} & D_{3}  \end{array} \\ \hline D_{4}   \end{array}\right)$, where
\[D_{1}=\left(\begin{array}{c|c}
 I_{l'} & \\\hline
& 
\end{array}\right)\in M(k-m-l, n-2k),\, D_{2}=\left(\begin{array}{c|c}
& \\\hline
I_{l''}& 
\end{array}\right)\in M(k-m-l, m),\, D_{3}=\left(\begin{array}{c|c}
 & \\\hline
& I_{l}
\end{array}\right)\in M(k-m, k-m)\]
for some $s-2m\leq l\leq k-m$ and $l'+l''=k-m-l$. Then, $x=ghw_{k-m}w_{0}y'(ghw_{k-m}w_{0})^{-1}$, where $h=w_{k-m}w_{0}\left(\begin{array}{c|c}
Q & \\
\hline  & P^{-1}
\end{array}\right)w_{0}w_{k-m}$ is a block lower triangular matrix in $C_{G}([V,t])$ and $y'=\left(\begin{array}{c|c} I_{n-k} & \\ \hline PBQ & I_{k} \end{array}\right)$.

Assume that either $l>s-2m$ or $l=s-2m$ with $s$ even. By multiplying a permutation $P'\in\GL_{k}(F),Q'\in \GL_{n-k}(F)$ of the form in (\ref{formofpq}) we change $D_{3}$ in $PBQ$ to
\begin{equation}
\label{D3PBQ}
\left(\begin{array}{c|c|c|c}
I_{p} &  & \\
\hline  & I_{\lceil \frac{s}{2} \rceil-m} & \\
\hline  &  & D_{3}'\\
\hline  &  &  & I_{\lceil \frac{s}{2} \rceil-m}
\end{array}\right)
\text{ with } p=\begin{cases}
l+2m-2\lceil \frac{s}{2} \rceil & \text{ if \ensuremath{l\leq k-2m}}\\
l+m-2\lceil \frac{s}{2} \rceil & \text{ if \ensuremath{l>k-2m} and \ensuremath{l\geq2\lceil \frac{s}{2} \rceil-m}}\\
0 & \text{ if \ensuremath{l>k-2m} and \ensuremath{l<2\lceil \frac{s}{2} \rceil-m}},
\end{cases}
\end{equation}
where 
\[
D_{3}'=\begin{cases}
0 & \text{ if \ensuremath{l\leq k-2m}}\\
I_{m} & \text{ if \ensuremath{l>k-2m} and \ensuremath{l\geq2\lceil \frac{s}{2} \rceil-m}}\\
\left(\begin{array}{c|c}
I_{l-2(\lceil \frac{s}{2} \rceil-m)}\\\hline
& 
\end{array}\right) & \text{ if \ensuremath{l>k-2m} and \ensuremath{l<2\lceil \frac{s}{2} \rceil-m}}.
\end{cases}
\]
We replace $ P $ by $ P'P $ and $Q$ by $QQ'$ in both $h$ and $y'$, which with abuse of notations we still denote by $h$ and $y'$. Consider an involution
\[z'=\left(\begin{array}{c|c|c} 
	I_{n-2k} & &\\ \hline
	 & I_{1}(k, \lceil \frac{s}{2} \rceil) & \\ \hline
	\begin{array}{c|c} I_{n-2k+s-2\lceil \frac{s}{2} \rceil} &\\ \hline \vspace{0.18cm} \,& \end{array}&  \begin{array}{c|c}  & \\ \hline I_{2m} & \end{array} & I_{2}(k, \lceil \frac{s}{2} \rceil)
\end{array}\right), \]
\[\text{ where }
I_{1}(k, \lceil \frac{s}{2} \rceil)=\left(\begin{array}{c|c|c}
I_{k-s} & \\
\hline  & I_{\lceil \frac{s}{2} \rceil} & \begin{array}{c|c}
 & \\
\hline  & I_{\lceil \frac{s}{2} \rceil-m}
\end{array}\\
\hline  &  & I_{\lceil \frac{s}{2} \rceil}
\end{array}\right)
\text{ and }\,\, I_{2}(k,\lceil \frac{s}{2} \rceil)=\left(\begin{array}{c|c|c}
I_{k-s} & \\
\hline  & I_{\lceil \frac{s}{2} \rceil} & \begin{array}{c|c}
I_{\lceil \frac{s}{2} \rceil-m} & \\
\hline  & 
\end{array}\\
\hline  &  & I_{\lceil \frac{s}{2} \rceil}
\end{array}\right).
\]

Then, by a direct calculation, we see that $y'\sim z'$. Since $ghw_{k-m}w_{0}z'(ghw_{k-m}w_{0})^{-1}$ is block lower triangular in $C_{G}[V,t]$, its distance from $t$ is at most two by Proposition \ref{lowertriangular}. Therefore, $d(t, x)\leq 3$.

Similarly, if $l=s-2m$ with $s$ odd, then the same proof works if we replace (\ref{D3PBQ}) and $z'$ by 
\[\left(\begin{array}{c|c|c|c|c}
 &  &  & & \\
\hline 
 & I_{[\frac{s}{2}]-m} &  & & \\
\hline 
&  &  & &\\
 \hline
& & & 1 &\\
\hline 
 &  & & & I_{[\frac{s}{2}]-m}
\end{array}\right)
\text{ and } 
\left(\begin{array}{c|c|c} 
I_{n-2k} & &\\ \hline
& I_{1}(k, \lceil \frac{s}{2} \rceil) & \\ \hline
\begin{array}{c} I_{n-2k} \\ \hline \vspace{0.18cm} \,\end{array}&  \begin{array}{c|c}  & \\ \hline I_{2m} & \end{array} & I_{2}'(k, [\frac{s}{2}])
\end{array}\right)
\]
where the middle zero block in the first one has size $m$ and $I_{2}'(k,[\frac{s}{2}]):=\left(\begin{array}{c|c|c|c}
	I_{k-s} &  & \\
	\hline  & I_{[\frac{s}{2}]} &  & \begin{array}{c|c}
		I_{[\frac{s}{2}]-m} & \\\hline
		& 
	\end{array}\\
	\hline  &  & 1\\
	\hline  &  &  & I_{[\frac{s}{2}]}
\end{array}\right)$, respectively.\end{proof}

Applying the same argument as used in the proof of Theorem \ref{even diam}, we determine the diameter of the commuting graph on the set of all involutions.

\begin{corollary}\label{allinvolution}
Let $F$ be a field of characteristic $2$, $n\geq 3$, and $X$ the set of all involutions in $G$. Then, $\diam\Gamma(G, X)=3$. 
\end{corollary}
\begin{proof}
If $n=3$, then $X=X_{1}$, thus the result follows from \cite[Theorem 3.1]{BBPR}. So we assume that $n\geq 4$. Let $x\in X_{k}$ for some $1\leq k\leq n/2$. In order to establish our upper bound, it suffices to show that $d(x, t'_{i})\leq 3$ for all $1\leq i\leq n/2$, where $t'_{i}=J(n,i)\in X_{i}$. Note that $t'_{i}\sim t'_{k}$ for all $i, k$. If $n\geq 4k$, then by \cite[Lemma 4.3]{BBPR}, we have $d(x, t'_{k})\leq 2$, thus $d(x,t'_{i})=d(x, t'_{k})+d(t'_{k},t'_{i})\leq 3$. Hence, we assume that $n/4<k\leq n/2$. It follows from the proof of Theorem \ref{even diam} that $x$ is either of the form (\ref{xeventwo}) for $2m<k$ or
\begin{equation}\label{formxcor}
x=gw_{m}\left(\begin{array}{c|c}
I_{n-k}\\
\hline B & I_{k}
\end{array}\right)w_{m}g^{-1}
\end{equation}
for $2m\geq k$. Observe that the middle matrices in (\ref{xeventwo}) and (\ref{formxcor}) commute with $t'_{1}$. Hence, it is enough to show that $d(x_{1}, t'_{i})\leq 2$ for any $x_{1}\in X_{1}$ and all $2\leq i\leq n/2$. Let $m_{1}=\dim([V,x_{1}]\cap [V,t_{1}])$. If $m_{1}=1$, then by (\ref{formxcor}) we have a path $x_{1}\sim t'_{1}\sim t'_{i}$. Otherwise, we have either $x_{1}=gw_{0}t'_{1}w_{0}g^{-1}$ or $x_{1}=gw_{0}t_{1}w_{0}g^{-1}$ for some $g\in C_{G}([V,t_{1}])$, i.e., 
\begin{equation}\label{formxone}
x_{1}=g\left(\begin{array}{c|c}
J(n-1,1)\\
\hline  & 1
\end{array}\right)g^{-1} \text{ or } x_{1}=g\left(\begin{array}{c|c}
I_{n-1} & \begin{array}{c} \\\hline 1 \end{array}\\
\hline  & 1
\end{array}\right)g^{-1} \text{ with } g=\left(\begin{array}{c|c}
P\\
\hline Q & r
\end{array}\right)
\end{equation}
for some $P\in \GL_{n-1}(F)$, $r\in F^{\times}$.

We shall find $y_{1}\in X_{1}$ such that $x_{1}\sim y_{1}\in C_{G}(t_{i}')$ for each case in (\ref{formxone}). Consider the first case in (\ref{formxone}). Write $P^{-1}=\left(\begin{array}{c|c}
P_{1} & P_{2}\\
\hline P_{3} & P_{4}
\end{array}\right)$ with $P_{2}\in \mat(i-1,i-1)$ and $P_{4}\in \mat(n-i,i-1)$. Let $M\in \GL_{n-1}(F)$ and $N\in \GL_{i-1}(F)$ such that
\begin{equation}\label{pmn}
\left(\begin{array}{c}
P_{2}\\
\hline P_{4}
\end{array}\right)=M\left(\begin{array}{c}
I_{i-1}\\
\hline \\
\end{array}\right)N.
\end{equation}
Find a nonzero row $B\in \mat(1,n-1)$ such that the first $i-1$ elements are all zero and the last element of $BM^{-1}$ is zero. Then, by (\ref{pmn}) the last $i-1$ elements of $BM^{-1}P^{-1}$ are all zero. Take
\[y_{1}=g\left(\begin{array}{c|c}
I_{n-1} & \\
\hline BM^{-1} & 1
\end{array}\right)g^{-1}=\left(\begin{array}{c|c}
I_{n-1}\\
\hline rBM^{-1}P^{-1} & 1
\end{array}\right).\]
By construction, this involution is contained in $C_{G}(t_{i}')$ and commutes with $x_{1}$.

Now we consider the second involution in (\ref{formxone}). By the argument given in the proof of Theorem \ref{even diam}, we can choose 
\[
M=\left(\begin{array}{c|c}
M_{1} & \\ \hline
 & I_{n-i-1}\end{array}\right), N=\left(\begin{array}{c|c}
 N_{1} & \\ \hline
 N_{2 }& N_{3}\end{array}\right)\in \GL_{n-1}(F), \,N_{1}\in \GL_{i}(F)
\]
such that
\[P':=M^{-1}PN^{-1}=\left(\begin{array}{c}
\begin{array}{c|c}
P_{1} & P_{2} \end{array}\\ \hline
P_{3}\end{array}\right) \text{ with } P_{1}=\left(\begin{array}{c|c}
I_{l} & \\ \hline & \end{array}\right)\in \mat(i, i), P_{2}=\left(\begin{array}{c|c}
& \\ \hline &I_{i-l} \end{array}\right)\in \mat(i, n-1-i)\]
for some $0\leq l\leq i$. First assume $ l>0. $ As $i<n-1$, there is a zero column in $\left(\begin{array}{c|c} P_{1} & P_{2}\end{array}\right)$, say $j$-th column. Take
\begin{equation}\label{yonetwo}
y_{1}=g\left(\begin{array}{c|c}
N^{-1}(I_{n-1}+e_{j1})N & \\
\hline  & 1
\end{array}\right)g^{-1}=\left(\begin{array}{c|c}
M & \\
\hline  & 1
\end{array}\right)\left(\begin{array}{c|c}
P'(I_{n-1}+e_{j1})P'^{-1}\\
\hline QN^{-1}e_{j1}P'^{-1} & 1
\end{array}\right)
\left(\begin{array}{c|c}
M^{-1}\\
\hline  & 1
\end{array}\right).
\end{equation}
Then, one can easily check that the second middle matrix in (\ref{yonetwo}) is contained in $C_{G}(t'_{i})$ and the first middle matrix in (\ref{yonetwo}) commutes with the middle matrix in (\ref{formxone}). Therefore, $d(x_{1},t'_{i})\leq 2$, thus $d(x,t'_{i})\leq 3$. If $l=0$, then $ i\leq n-1-i$ and $\left(\begin{array}{c|c} P_{1} & P_{2}\end{array}\right)=\left( \begin{array}{c|c}&I_{i}\end{array}\right) $. Similarly, we can choose
\[M'=\left(\begin{array}{c|c}
I_{i} & \\ \hline
M_{2}& \begin{array}{c|c} M_{3} & \\ \hline M_{4}& I_{i-1}\end{array}\end{array}\right), N'=\left(\begin{array}{c|c}
\begin{array}{c|c} N_{4} & \\ \hline N_{5} & N_{6} \end{array}& \\ \hline
& I_{i}\end{array}\right)\in \GL_{n}(F),\, M_{3}\in \GL_{n-2i}(F), N_{4}\in \GL_{i}(F)
\]
such that
\[P'':=M'^{-1}P'N'^{-1}=\left(\begin{array}{c|c}
& I_{i} \\ \hline
\begin{array}{c}
\begin{array}{c|c}
P'_{1} & P'_{2} \end{array}\\ \hline
P'_{3}\end{array} &\end{array}\right) \text{ with } P'_{1}=\left(\begin{array}{c|c}
I_{s} & \\ \hline & \end{array}\right)\in \mat(n-2i,i), P'_{2}=\left(\begin{array}{c|c}
& \\ \hline &I_{n-2i-s} \end{array}\right)\]
for some $0<s\leq n-2i$ ($s\neq 0$ as $P_{2}'\in \mat(n-2i,n-2i-1)$). Observe that the first entries of the last $i-1$ columns of $\left(\begin{array}{c}
	\begin{array}{c|c}
		P'_{1} & P'_{2} \end{array}\\ \hline
	P'_{3}\end{array}\right)^{-1}$ are all zero. 
By replacing $M, N, P'$ in (\ref{yonetwo}) by $MM', N'N, P''$, respectively, we get the same result as in the previous case $l>0$.

To prove the lower bound, consider the following involutions
\[x=\left(\begin{array}{c|c}
I_{k} & B\\
\hline  & I_{k}
\end{array}\right) \text{ if } n=2k, \text{ and } x=\left(\begin{array}{c|c|c}
I_{k} &  & B\\
\hline  & 1\\
\hline  &  & I_{k}
\end{array}\right) \text{ if } n=2k+1,\]
where $B$ is an invertible matrix that does not commute with any $J(k,r)$ for all $1\leq r\leq k/2$ as in Proposition \ref{even diam lower}. We show that $d(x, t'_{k})\geq 3$. Suppose that there exists an involution $y$ such that $x\sim y$. Then, $y$ has the following forms
\[
y=\left(\begin{array}{c|c}
A & C\\
\hline  & B^{-1}AB
\end{array}\right)   \text{ and }
y=\left(\begin{array}{c|c|c}
A &  & C\\
\hline  & d & E\\
\hline  &  & B^{-1}AB
\end{array}\right).
\]
for some $A\in \GL_{k}(F)$ and $d\in F^{\times}$ in the corresponding cases. If $y\in C_{G}(t'_{k})$, then $C=E=0$ and $BA=AB$ in both cases, which implies $y$ is the identity. This establishes the lower bound.\end{proof}

\section{Four-dimensional linear groups}\label{fdsec}
In the present section, we determine the structure of $\GL_{4}(F)$ over a finite field in Propositions \ref{fourx1x3} and \ref{fourx2}. In particular, by applying the results in the previous section we provide
the number of involutions $\Delta_{i}(t)$ for each distance $i$.

\begin{proposition}\label{fourx1x3}
	Let $G=\GL_{4}(F)$ over a finite field $F$ with $q$ elements and let $X_{k}$ be the class of involutions for $k=1,3$ in case $\Ch(F)\neq 2$ and $k=1$ in case $\Ch(F)=2$. Then, for any such $k$, $\diam\Gamma(G, X_{k})=2$. Moreover, 
	\begin{align*}
		|\Delta_{1}(t)|= & \begin{cases}
			q^4+2q^3-q^2-q-2 & \text{ if } \Ch(F)=2,\\
			q^{4}+q^{3}+q^{2} & \text{ if }\Ch(F)\neq2.
		\end{cases}\\
		|\Delta_{2}(t)|= & \begin{cases}
			q^6+q^5-2q^3 & \text{ if } \Ch(F)=2,\\
			q^{6}+q^{5}-q^{2}-1 & \text{ if } \Ch(F)\neq2.
		\end{cases}
	\end{align*}
\end{proposition}
\begin{proof}
	It follows from (\ref{deltaone}) that
	\[|X_{1}|=\begin{cases}
	|\GL_{4}(F)|/q^{5}|\GL_{2}(F)||\GL_{1}(F)|=(q^{4}-1)(q^{2}+q+1) & \text{ if } \Ch(F)=2,\\
	|\GL_{4}(F)|/|\GL_{3}(F)||\GL_{1}(F)|=q^{3}(q^{2}+1)(q+1)&\text{ if }\Ch(F)\neq2.
	\end{cases}
	\]
	By Theorem \ref{mainthm} (i) and a symmetric argument, it suffices to compute $|\Delta_{1}(t)|$ for $k=1$. If $\Ch(F)\neq2$, then by (\ref{deltaone}) we see that $\Delta_{1}(t)$ consists of involutions $\left(\begin{array}{c|c}
	A\\
	\hline  & 1
	\end{array}\right)$ such that $A$ is an involution in $\GL_{3}(F)$ with $k=1$. Hence, $|\Delta_{1}(t)|=q^{4}+q^{3}+q^{2}$.

	isume that $\Ch(F)=2$. Then, it follows by (\ref{deltaone}) that $\Delta_{1}(t)$ consists of the following involutions
	\[\left(\begin{array}{c|c|c}
	A & B\\
	\hline  & 1\\
	\hline C & d & 1
	\end{array}\right)\]
	such that $A^{2}=I_{2}, AB=B, CA=C, CB=0$. If $A=I_{2}$, then by using $\rank(I_{4}+x)=1$ one can easily check that the possible number of $x$ is $2q^{3}-q-2$. Now assume that $ A=PI(2,1)P^{-1}$ for some $P\in\GL_{2}(F)$. Then, the involution $x$ is decomposed as
	\begin{equation}
	\label{decompositionxdim4}
	\left(\begin{array}{c|c}
	P\\\hline
	& I_{2}
	\end{array}\right)\left(\begin{array}{cc|cc}
	1 & & & \\  1 & 1&b& \\\hline
	&& 1& \\  c & & d &1
	\end{array}\right)\left(\begin{array}{c|c}
	P^{-1}\\\hline
	& I_{2}
	\end{array}\right)
	\end{equation}
	for some $b,c\in F$ with $bc=d$. If $b=c=0$, then the possible number of $x$ is $q^{2}-1$. If either $b=0$ and $c\neq 0$ or $b\neq 0$ and $c=0$, then the number of $\left(\begin{array}{c|c}
	P\\\hline
	& I_{2}
	\end{array}\right)$-conjugacy class of the middle involution in (\ref{decompositionxdim4}) is $(q^{2}-1)(q-1)$, so is the possible number of $x$ in each case. Similarly, if $b\neq 0$ and $c\neq 0$, then the possible number of $x$ is $(q^{2}-1)(q-1)^{2}$, thus 
	\[ |\Delta_{1}(t)|=q^{4}+2q^{3}-q^{2}-q-2,
	\] 
	which completes the proof.\end{proof}

Now we consider the conjugacy class $X_{2}$. By the same argument as in the proof of Proposition \ref{fourx1x3}, we have
\begin{equation}\label{fourx2card}
|X_{2}|=\begin{cases}
q(q^{4}-1)(q^{3}-1) & \text{ if } \Ch(F)=2,\\
q^{4}(q^{2}+1)(q^{2}+q+1) & \text{ if } \Ch(F)\neq 2.
\end{cases}
\end{equation}

\begin{lemma}\label{strucUi}
Let $U_{i}=\{x\in X_{2}\mid\dim([V,t]\cap[V,x])=i\}$ for $0\leq i\leq 2$. Then,
\begin{align*}
|U_{0}|&=q^{4}(q^{2}-1)(q^{2}-q),& |U_{1}|&=(q^{2}-1)^{2}(q^{2}+q^{3}),& |U_{2}|&=
	(q^{2}-1)(q^{2}-q)-1 & \text{ if } \Ch(F)=2,\\
|U_{0}|&=q^{8},& |U_{1}|&=q^{5}(q+1)^{2},& |U_{2}|&=
q^{4}-1 & \text{ if } \Ch(F)\neq 2.
\end{align*}
\end{lemma}
\begin{proof}
The result for $U_{2}$ immediately follows from Lemma \ref{vxvteq}. By (\ref{fourx2card}), it suffices to compute $|U_{0}|$. We write $g\in C_{G}([V,t])$ as
\begin{equation}\label{ghr}
g=\left(\begin{array}{c|c}
	P\\
	\hline  & Q
\end{array}\right)\left(\begin{array}{c|c}
	I_{2}\\
	\hline R & I_{2} 
\end{array}\right) \text{ with }P, Q\in \GL_{2}(F), R\in \mat(2,2).\end{equation}
We denote by $h$ the block diagonal matrix in (\ref{ghr}). Then, by the proof of Theorem \ref{odd diam} and (\ref{xeventwo}) we have
\[
x=\begin{cases}
g\left(\begin{array}{c|c}
I_{2} & B\\
\hline  & I_{2}
\end{array}\right)g^{-1}=h\left(\begin{array}{c|c}
I_{2}+BR & B\\
\hline  RBR & I_{2}+RB
\end{array}\right)h^{-1} & \text{ if }\Ch(F)=2,\\
g\left(\begin{array}{c|c}
-I_{2} & B\\
\hline  & I_{2}
\end{array}\right)g^{-1}=h\left(\begin{array}{c|c}
-I_{2}-BR & B\\
\hline  -2R-RBR & I_{2}+RB
\end{array}\right)h^{-1} & \text{ if }\Ch(F)\neq2,
\end{cases}
\]
where $B$ is invertible in case $\Ch(F)=2$. Replacing $B$ and $R$ by $PBQ^{-1}$ and $QRP^{-1}$, respectively, we may assume that $h=I_{4}$. Hence, the possible number of $x$ is the product of possible numbers of $B$ and $R$, thus the result follows. \end{proof}

By Lemma \ref{vxvteq}, we see that $|U_{2}|=|\Delta_{1}(t)\cap U_{2}|$ if $\Ch(F)=2$ and $|U_{2}|=|\Delta_{2}(t)\cap U_{2}|$ otherwise. Hence, it remains to analyze the structures of $U_1$ and $U_0$.
	
\begin{lemma}\label{strucU1}
If $\Ch(F)=2$, then
	\begin{align*}
		|\Delta_{1}(t)\cap U_{1}|&=q^{2}(q^{2}-1),& |\Delta_{2}(t)\cap U_{1}|&=q^2(q^2-1)(2q^2-q-2),& |\Delta_{3}(t)\cap U_{1}|&=
		q^4(q^3-q^2-q+1).
	\end{align*}

Otherwise, we have
\begin{align*}
	|\Delta_{1}(t)\cap U_{1}|&=q^{2}(q+1)^{2},& |\Delta_{2}(t)\cap U_{1}|&=q(q+1)^{2}(q-1)(q^{3}+q+1),& |\Delta_{3}(t)\cap U_{1}|&=
	q(q+1)^{3}(q-1)^{2}.\end{align*}
\end{lemma}
\begin{proof}
We first assume that $\Ch(F)=2$. Let $x\in U_{1}$ and $H=C_{G}([V,t])$. Then, it follows from the proof of Theorem \ref{even diam} and Lemma \ref{vxvteq} that $x=hw_{1}\left(\begin{array}{c|c}
I_{2}\\
\hline B & I_{2}
\end{array}\right)w_{1}h^{-1}$ for some $h\in H$ and $B\in \GL_{2}(F)$. If $B$ is lower triangular, then one can find $h'\in H$ such that $x=h'w_{1}tw_{1}h'^{-1}$. Otherwise, one finds $h''\in H$ such that $x=h''w_{1}\left(\begin{array}{c|c}
I_{2 }& \\
\hline B' & I_{2} 
\end{array}\right)w_{1}h''^{-1}$, where $B'=\left(\begin{array}{c|c}
& 1\\
\hline 1 & 
\end{array}\right)$. Hence, $U_{1}=Hy_{1}H^{-1}\cup Hy_{2}H^{-1}$, where 
\[y_{1}=\left(\begin{array}{c|c}
I(2,1)\\
\hline  & I(2,1)
\end{array}\right) \text{ and } y_2=\left(\begin{array}{c|c}
I_{2} & e_{21}\\
\hline e_{21} & I_{2}
\end{array}\right).\]
By a simple calculation, we have $|Hy_{1}H^{-1}|=q^{2}(q^{2}-1)^{2}$ and $|Hy_{2}H^{-1}|=q^{3}(q^{2}-1)^{2}$.

Let $x\in Hy_{1}H^{-1}$. Then, by Proposition \ref{lowertriangular}, $x\in \Delta_{1}(t)\cup \Delta_{2}(t)$, thus it suffices to compute $|\Delta_{1}(t)\cap Hy_{1}H^{-1}|$. Let $x=hy_{1}h^{-1}\in \Delta_{1}(t)\cap Hy_{1}H^{-1}$ with 
$h=\left(\begin{array}{c|c}
P & \\
\hline R & Q
\end{array}\right)\in H$. Then, by (\ref{deltaone}) $PQ^{-1}\in C_{\GL_{2}(F)}(I(2,1))$, thus we see that $x\in Ky_{1}K^{-1}$, where $K=\{\left(\begin{array}{c|c}
P & \\
\hline R & P
\end{array}\right) \,|\, P\in \GL_{2}(F), R\in \mat(2,2)\}$. Therefore, we obtain 
\begin{equation}\label{U1y1}
|\Delta_{1}(t)\cap Hy_{1}H^{-1}|=q^{2}(q^{2}-1), \text{ thus } |\Delta_{2}(t)\cap Hy_{1}H^{-1}|=q^{2}(q^{4}-3q^{2}+2).
\end{equation}
Let $x\in Hy_{2}H^{-1}$. Then, by (\ref{deltaone}) we see that $x\in \Delta_{2}(t)\cup \Delta_{3}(t)$, thus it suffices to calculate $|\Delta_{2}(t)\cap Hy_{2}H^{-1}|$. If $x\in \Delta_{2}(t)\cap Hy_{2}H^{-1}$, then there exists $z=h\left(\begin{array}{c|c}
E_{21}(a) & \\
\hline be_{21} & E_{21}(c)
\end{array}\right)h^{-1}\in \Delta_{1}(t)$ for some $a,c\in F^{\times}, b\in F$ and $h=\left(\begin{array}{c|c}
P & \\
\hline R & Q
\end{array}\right)\in H$. Moreover, by (\ref{deltaone}) we have $PQ^{-1}=\left(\begin{array}{c|c}
	d & \\
	\hline ce & a^{-1}cd
\end{array}\right)$ for some $d\in F^{\times}$ and $e\in F$. Hence, we see that $x\in Ky_{3}K^{-1}$, where $y_{3}=\left(\begin{array}{c|c}
I_{2} & de_{21} \\
\hline  a(cd)^{-1}e_{21}& I_{2}
\end{array}\right)$. Therefore, we obtain that
\begin{equation*}
|\Delta_{2}(t)\cap Hy_{2}H^{-1}|=q^{3}(q-1)^{2}(q+1), \text{ thus } |\Delta_{3}(t)\cap Hy_{2}H^{-1}|=q^4(q^3-q^2-q+1),
\end{equation*}
which together with (\ref{U1y1}) completes the proof of case $\Ch(F)=2$.

Now assume that $\Ch(F)\neq 2$. It is obvious that $|\Delta_{1}(t)\cap U_{1}|=q^{2}(q+1)^{2}$. Hence, it is enough to calculate $|\Delta_{3}(t)\cap U_{1}|$. Let $H'$ be the subgroup of $H$ consisting of block diagonal matrices and $x\in U_{1}$. Then, by the proof of Theorem \ref{odd diam} together with (\ref{ghr}), $x$ is decomposed as
\begin{equation}\label{decompxU1charnot2}
\left(\begin{array}{c|c}
I_{2} &  \\
\hline  R& I_{2}
\end{array}\right)\left(\begin{array}{c|c}
I(2,1) & be_{21} \\
\hline  & I(2,1)
\end{array}\right)\left(\begin{array}{c|c}
I_{2} &  \\
\hline  -R& I_{2}
\end{array}\right)=\left(\begin{array}{c|c}
I(2,1)-be_{21}R & be_{21} \\
\hline  RI(2,1)-(I(2,1)+bRe_{21})R & bRe_{21}+I(2,1)
\end{array}\right)
\end{equation}
for some $b\in F$ and $R\in \mat(2,2)$, up to conjugation in $H'$. If $b=0$, then by Proposition \ref{lowertriangular}, $d(t,x)\leq 2$, thus we may assume that $b\neq 0$. Consider the following element $y\in \Delta_{1}(t)$ 
\[y=h\left(\begin{array}{c|c}
I_{2} &  \\
\hline  R& I_{2}
\end{array}\right)\left(\begin{array}{c|c}
A & B \\
\hline C & D
\end{array}\right)\left(\begin{array}{c|c}
I_{2} &  \\
\hline  -R& I_{2}
\end{array}\right)h^{-1}\]
for some $h\in H'$, $A,B,C,D\in \mat(2,2)$ such that $x\sim y$. Then, we obtain that
\[B=0, C=\left(\begin{array}{c|c}
c_{1} &  \\
\hline  & c_{2}
\end{array}\right), A=\left(\begin{array}{c|c}
\pm 1 &  \\
\hline  bc_{1}/2 & \mp 1
\end{array}\right), D=\left(\begin{array}{c|c}
\mp 1 &  \\
\hline  -bc_{2}/2 & \pm 1
\end{array}\right)\]
for some $c_{1}, c_{2}\in F$ and
\[c_{1}(2+br_{2})=\mp 4r_{1},\, c_{2}(2+br_{2})=\pm 4r_{4},\, b(c_{2}r_{1}+c_{1}r_{4})=0\, \text{ with } R=\left(\begin{array}{c|c}
r_{1}& r_{2}  \\
\hline r_{3} & r_{4}
\end{array}\right).
\]
Hence, we see that $x\in \Delta_{3}(t)$ is equivalent to the conditions $b\neq 0$, $br_{2}=-2$, and either $r_{1}\neq 0$ or $r_{4}\neq 0$. Therefore, it follows from (\ref{decompxU1charnot2}) that 
\begin{equation}\label{formofxU1}
x=\left(\begin{array}{cc|cc}
1 & & &\\ 
-br_{1} & 1 & b &\\ \hline
2r_{1} & & -1 &\\ 
2r_{3}-br_{1}r_{4} & 2r_{4} & br_{4} & -1
\end{array}\right),
\end{equation}
up to conjugation in $H'$ if $x\in \Delta_{3}(t)$. Using (\ref{formofxU1}), we see that $\Delta_{3}(t)\cap U_{1} $ is decomposed into the following union of three orbits 
\[H'\left(\begin{array}{cc|cc}
1&\\
& 1 & \,\,1\\\hline
&  & -1\,\,\\
& 2 & \,\,1 & -1
\end{array}\right)H'^{-1}\cup H'\left(\begin{array}{cc|cc}
\,\,1&\\
-1\,\, & 1 & 1\\\hline
\,\,2 &  & -1\\
&  &  & -1
\end{array}\right)H'^{-1}\cup H'\left(\begin{array}{cc|cc}
\,\,1&\\
-1\,\, & 1 & \,\,1\\\hline
\,\,2 &  & -1\,\,\\
& 2 & \,\,1 & -1
\end{array}\right)H'^{-1}\]
with orders $q(q^{2}-1)^{2}$, $q(q^{2}-1)^{2}$, $q(q-1)(q^{2}-1)^{2}$, respectively. Therefore, we obtain $|\Delta_{3}(t)\cap U_{1}|=q(q+1)^{3}(q-1)^{2}$ and $|\Delta_{2}(t)\cap U_{1}|=q(q+1)^{2}(q-1)(q^{3}+q+1)$.\end{proof}

\begin{lemma}\label{strucU0}
	If $\Ch(F)=2$, then
	\begin{align*}
		|\Delta_{1}(t)\cap U_{0}|&=0,& |\Delta_{2}(t)\cap U_{0}|&=q^{6}(q-1),& |\Delta_{3}(t)\cap U_{0}|&=
		q^{5}(q-1)(q^{2}-q-1).
	\end{align*}
	
	Otherwise, we have $|\Delta_{1}(t)\cap U_{0}|=1$,
	\begin{align*}
		|\Delta_{2}(t)\cap U_{0}|&=\frac{1}{2}(q-1)(q+1)^{2}(q^{5}-q^{3}+2q^{2}+2),& |\Delta_{3}(t)\cap U_{0}|&=
		\frac{1}{2}q(q-1)^{2}(q+1)(q^{4}+3q^{2}+2q+2).\end{align*}
\end{lemma}
\begin{proof}
First, consider the case $\Ch(F)=2$. Let $x\in U_{0}$. As in the proof of Lemma \ref{strucU0} we may assume that 
\[x=\left(\begin{array}{c|c}
I_{2}+BR & B\\
\hline  RBR & I_{2}+RB
\end{array}\right)=\left(\begin{array}{c|c}
I_{2}\\\hline
R & I_{2}
\end{array}\right)\left(\begin{array}{c|c}
I_{2} & B\\\hline
& I_{2}
\end{array}\right)\left(\begin{array}{c|c}
I_{2}\\\hline
R & I_{2}
\end{array}\right),\]
where $B\in \GL_{2}(F)$ and $R\in \mat(2,2)$. If $x\in\Delta_{1}(t)$, then $B=0$, i.e., $x=I_{4}$, which implies $|\Delta_{1}(t)\cap U_{0}|=0$. We shall compute $|\Delta_{2}(t)\cap U_{0}|$. Observe that for $y\in \Delta_{1}(t)$
\begin{equation}\label{xycommutecond}
x\sim y:=\left(\begin{array}{c|c}
M\\\hline
N & M
\end{array}\right)  \Leftrightarrow
 MB=BM,\, N=RM+MR.
\end{equation}
As $M^{2}=I_{2}$ and $y$ is an involution, it follows from the first equation in (\ref{xycommutecond}) that 
\begin{equation}\label{MBexpre}
M=gI(2,1)g^{-1} \text{ and }  B=agE_{21}(b)g^{-1}
\end{equation}
for some $g\in \GL_{2}(F)$, $a\in F^{\times}$, and $b\in F$. Moreover, for any such $B$ and $M$ and any $R=g\left(\begin{array}{cc}
r_{1} & r_{2}\\
r_{3} & r_{4}
\end{array}\right)g^{-1}\in \mat(2,2)$, there exists $N:=g\left(\begin{array}{cc}
r_{2} & \\
r_{1}+r_{4} & r_{2}
\end{array}\right)g^{-1}$ satisfies the second equation in (\ref{xycommutecond}). Hence, the possible number of $x$ is the product of possible numbers of $B$ as in (\ref{MBexpre}) and $R$, which is $q^{2}(q-1)$ and $q^{4}$, respectively. Therefore, we obtain that $|U_{0}\cap\Delta_{2}(t)|=q^{6}(q-1)$ and $|U_{0}\cap\Delta_{3}(t)|=q^{5}(q-1)(q^{2}-q-1)$.

Now we assume that $\Ch(F)\neq 2$. Similarly, we may assume that
\[x=\left(\begin{array}{c|c}
	-I_{2}-BR & B\\
	\hline  -2R-RBR & I_{2}+RB
\end{array}\right)
=\left(\begin{array}{c|c}
I_{2}\\\hline
R & I_{2}
\end{array}\right)\left(\begin{array}{c|c}
-I_{2} & B\\\hline
& I_{2}
\end{array}\right)\left(\begin{array}{c|c}
I_{2}\\\hline
-R & I_{2}
\end{array}\right),\]
where $B, R\in \mat(2,2)$. If $x\in \Delta_{1}(t)$, then $R=B=0$, thus $|\Delta_{1}(t)\cap U_{0}|=1$. We assume that at least one of $R$ and $B$ is nonzero. We shall find $|\Delta_{2}(t)\cap U_{0}|$. If $B=0$ and $R\neq 0$ or $B\neq 0$ and $R=0$, then $x$ is of the following form
\begin{equation}\label{formBR}
x=\left(\begin{array}{c|c}
-I_{2} & \\
\hline -2R & I_{2}
\end{array}\right)\text{ or }\left(\begin{array}{c|c}
-I_{2} & B\\
\hline  & I_{2}
\end{array}\right),
\end{equation}
respectively. It follows by Proposition \ref{lowertriangular} that they are all in $\Delta_{2}(t)$ and the number of possible $x$ is $q^{4}-1$ for each case. From now on we assume that $B\neq 0$ and $R\neq 0$.

If $y\in \Delta_{1}(t)$ such that $x\sim y$, then the involution $y$ is of the following form
\[y=\left(\begin{array}{c|c}
I_{2}\\\hline
R & I_{2}
\end{array}\right)\left(\begin{array}{c|c}
M & \\\hline
& N
\end{array}\right)\left(\begin{array}{c|c}
I_{2}\\\hline
-R & I_{2}
\end{array}\right)\]
for some involutions $M, N$ such that $RM=NR$ and $BN=MB$. Let $M=gI(2,1)g^{-1}$ and $N=hI(2,1)h^{-1}$ for some $g,h\in \GL_{2}(F)$. Then, we have
\[B=g\left(\begin{array}{cc}
a\\
& b
\end{array}\right)h^{-1} \text{ and } R=h\left(\begin{array}{cc}
c\\
& d
\end{array}\right)g^{-1}
\]
for some $a, b, c, d\in F$. Hence, if $x\in \Delta_{2}(t)$, then $x$ is of the following form
\[x=\left(\begin{array}{cc|cc}
-1+ac &  & a\\
& -1+bd &  & b\\\hline
-2c-ac^{2} &  & 1+ac\\
& -2d-bd^{2} &  & 1+bd
\end{array}\right),\]
up to conjugation in the subgroup $H'$ of block diagonal matrices. Using this, we obtain an $H'$-orbit decomposition with the following representatives
\[\left(\begin{array}{cc|cc}
-1 &  & \\
& \lambda-1 &  & 1\\\hline
&  & 1\\
& -\lambda^{2}-2\lambda &  & \lambda+1
\end{array}\right),
\left(\begin{array}{cc|cc}
-1 &  & \\
& \lambda-1 &  & 1\\\hline
-2 &  & 1\\
& -\lambda^{2}-2\lambda &  & \lambda+1
\end{array}\right),
\left(\begin{array}{cc|cc}
\lambda-1 &  & 1\\
& \lambda-1 &  & 1\\\hline
-\lambda^{2}-2\lambda &  & \lambda+1\\
& -\lambda^{2}-2\lambda &  & \lambda+1
\end{array}\right)
\]
with $\lambda\neq 0$ and 
\[\left(\begin{array}{cc|cc}
-1 &  & 1\\
& -1 &  & \\\hline
&  & 1\\
& -2 &  & 1
\end{array}\right),\, \left(\begin{array}{cc|cc}
-1+\lambda_{1} &  & 1\\
& -1+\lambda_{2} &  & 1\\\hline
-\lambda_{1}(2+\lambda_{1}) &  & 1+\lambda_{1}\\
& -\lambda_{2}(2+\lambda_{2}) &  & 1+\lambda_{2}
\end{array}\right)\]
with $\lambda_{1}\neq \lambda_{2}$. Then, the total order of all the orbits is
\[q^{2}(q^{2}-1)^{2}+q^{2}(q-1)^{3}(q+1)^{2}+q(q^{2}-1)(q-1)^{2}+(q^{2}-1)^{2}+\frac{q^{3}(q-1)^{3}(q+1)^{2}}{2}.\]
Hence, together with case (\ref{formBR}) we have
\[|\Delta_{2}(t)\cap U_{0}|=\frac{1}{2}(q-1)(q+1)^{2}(q^{5}-q^{3}+2q^{2}+2), |\Delta_{3}(t)\cap U_{0}|= \frac{1}{2}q(q-1)^{2}(q+1)(q^{4}+3q^{2}+2q+2).
\]\end{proof}

Combining Lemmas \ref{strucUi}, \ref{strucU1}, and \ref{strucU0} with Theorem \ref{mainthm} immediately yields the following result.
\begin{proposition}\label{fourx2}
		Let $G=\GL_{4}(F)$ over a finite field $F$ with $q$ elements. Then, $\diam\Gamma(G, X_{2})=3$. In particular, 
		\begin{align*}
			|\Delta_{1}(t)|= & \begin{cases}
				2q^4-q^3-2q^2+q-1 & \text{ if } \Ch(F)=2,\\
				q^2(q+1)^2+1 & \text{ if }\Ch(F)\neq2.
			\end{cases}\\
			|\Delta_{2}(t)|= & \begin{cases}
				q^2(q^5+q^4-q^3-4q^2+q+2) & \text{ if } \Ch(F)=2,\\
				\frac{1}{2}(q-1)(q+1)(q^6+3q^5+q^4+3q^3+8q^2+4q+4) & \text{ if } \Ch(F)\neq2.
			\end{cases}\\
				|\Delta_{3}(t)|= & \begin{cases}
					q^4(q^4-q^3-q^2+1) & \text{ if } \Ch(F)=2,\\
					\frac{1}{2}q(q-1)^2(q+1)(q^4+5q^2+6q+4) & \text{ if } \Ch(F)\neq2.
				\end{cases}
		\end{align*}
\end{proposition}

\paragraph{\bf Acknowledgments.}
This work was partially supported by National Research Foundation of Korea (NRF) funded by the Ministry of Science, ICT and Future Planning (2016R1C1B2010037).

\end{document}